\documentclass[reqno]{amsart}

\usepackage[round]{tsnatbib}

\usepackage{mathabx}

\usepackage{enumerate}

\usepackage{amsmath, amsthm}

\usepackage{amsmath, amsthm}
\usepackage{amsfonts,amssymb,dsfont}
\usepackage{graphicx}
\usepackage{pgf}
\usepackage{nicefrac,booktabs,mathrsfs}
\usepackage{bm} 
\usepackage{mathtools}                                

\usepackage[backgroundcolor=white,bordercolor=orange]{todonotes}

\renewcommand{\P}{\mathbb{P}}

\newcommand{\ccL}{{\mathscr L}}
\newcommand{\ccB}{{\mathscr B}}

\newcommand{\ccD}{{\mathscr D}}
\newcommand{\ccF}{{\mathscr F}}\newcommand{\cF}{{\mathcal F}}
\newcommand{\ccG}{{\mathscr G}}
\newcommand{\ccH}{{\mathscr H}}
\newcommand{\ccP}{{\mathscr P}}

\newcommand{\cN}{{\mathcal N}}
\newcommand{\ccS}{{\mathscr S}}

\newcommand{\ccR}{{\mathscr R}}

\renewcommand{\P}{\mathbb{P}}

\newcommand{\Ind}{{\mathds 1}}
\newcommand{\ind}[1]{\Ind_{\{#1\}}}

\newcommand{\R}{\mathbb{R}}

\newcommand{\bbF}{\mathbb{F}}

\newcommand{\N}{\mathbb{N}}

\newcommand{\PP}{\mathbb{P}}
\renewcommand{\P}{\PP}

\newcommand{\be}{\bm{e}}

\newcommand{\bA}{\bm{A}}
\newcommand{\bL}{\bm{L}}

\newcommand{\bH}{\bm{H}}
\newcommand{\bP}{\bm{P}}

\newcommand{\scal}[2]{\left\langle{#1},{#2}\right\rangle} 

\newtheorem{theorem}{Theorem}[section]
\newtheorem{corollary}[theorem]{Corollary}      
\newtheorem{lemma}[theorem]{Lemma}              
\newtheorem{proposition}[theorem]{Proposition}  

\theoremstyle{definition}

\newtheorem{definition}{Definition}[section]
\newtheorem{remark}{Remark}[section]

\usepackage[pdftex,colorlinks,urlcolor=red,citecolor=blue,linkcolor=red]{hyperref}

\renewcommand{\cF}{\ccF}

 \usepackage{amsmath,relsize}  

\RequirePackage{colortbl}
\definecolor{ts}{rgb}{1.0,0.6,0.0}
\definecolor{ma}{rgb}{1.0,0.8,0.0}  

\RequirePackage{colortbl}
\definecolor{tscolor}{rgb}{1.0,0.6,0.0}
\usepackage[ulem=normalem]{changes}      
\definechangesauthor[name = Thorsten , color=tscolor]{ts}  
\definechangesauthor[name = Maria, color=tscolor]{ma}  

\newcommand{\bk}{\mathbf{k}}
\newcommand{\bl}{\mathbf{l}}

\DeclareFontEncoding{FMS}{}{}
\DeclareFontSubstitution{FMS}{futm}{m}{n}
\DeclareFontEncoding{FMX}{}{}
\DeclareFontSubstitution{FMX}{futm}{m}{n}
\DeclareSymbolFont{fouriersymbols}{FMS}{futm}{m}{n}
\DeclareSymbolFont{fourierlargesymbols}{FMX}{futm}{m}{n}
\DeclareMathDelimiter{\VERT}{\mathord}{fouriersymbols}{152}{fourierlargesymbols}{147}
  
\begin{document}

\title{Time-inhomogeneous polynomial processes}
    \author[M. Agoitia]{Mar\'ia Fernanda del Carmen Agoitia Hurtado}
    \address{Chemnitz University of Technology, Reichenhainer Str. 41, 09126 Chemnitz, Germany.}
		\author[T. Schmidt]{Thorsten Schmidt}
		\address{Albert-Ludwigs University of Freiburg, Eckertstr. 1, 79104 Freiburg, Germany.}
    \email{ thorsten.schmidt@stochastik.uni-freiburg.de}
    \date{\today. We thank Martin Larsson for his insightful suggestions. The first author thanks the Consejo Nacional de Ciencia y Tecnolog\'ia (CONACyT) for generous support. }

\maketitle

\vspace{2mm}

\begin{abstract}
Time homogeneous polynomial processes are Markov processes whose moments can be calculated easily through matrix exponentials. In this work, we develop a notion of time inhomogeneous polynomial processes where the coeffiecients of the process may depend on time. A full characterization of this model class is given by means of their semimartingale characteristics. We show that in general, the computation of moments by matrix exponentials is no longer possible. 
As an alternative we explore a connection to Magnus series for fast numerical approximations. 

Time-inhomogeneity is important in a number of applications: in term-structure models, this allows a perfect calibration to available prices. In electricity markets, seasonality comes naturally into play and have to be captured by the used models.  The model class studied in this work extends existing models, for example Sato processes and time-inhomogeneous affine processes. 
\end{abstract}


\section{Introduction}
Many applications of Markov processes, in particular in mathematical finance, profit from an efficient computability of moments. This is an important feature for the computation of option prices and moment estimators, for example.
This motivates the study of \emph{polynomial processes}, introduced in \cite{CuchieroKellerResselTeichmann2012}, as they satisfy this  property directly by definition: a polynomial process $X$ is a Markov process, such that the expected value of a polynomial of the process at a future time  is given by a polynomial (of the same degree) of the initial value of the process.

It turns out that this property is shared by the well-known \emph{affine processes}, given the existence of the treated moments. Further important examples classify as polynomial processes:  
\emph{Jacobi processes}, which are processes living on a compact subset of the whole space, see for example \cite{DelbaenShirakawa02} and \cite{GourierouxJasiak06}, and 
\emph{Pearson diffusions}, exhibiting a linear drift and a squared diffusion coefficient and being studied in 
\cite{FormanSorensen08}.

Besides this, polynomial processes have gained increasing popularity in finance, see \cite{FilipovicLarsson2016}. The fields of applications range from term-structure modelling (see \cite{ChengTechranchi15} and \cite{Grbac15}) to variance swaps (see \cite{FilipovicGourierMancini14}), LIBOR models (see \cite{GlauGrbacKellerRessel2014}) and variance reduction for option pricing and hedging as in  \cite{CuchieroKellerResselTeichmann2012}.

All the above models are homogeneous in time and therefore do not allow for seasonality, for example. \emph{Time inhomogeneity} does also play an important role in term-structure modelling as it allows for perfect calibration to the observed market data. 
A natural example where all the features of time-inhomogeneous polynomial comes into play, namely bounded state space and seasonality, are spot markets of electricity and we refer to \cite{Agoitia2017} for a detailed application.

It is the purpose of this paper to characterize and study time-inhomogeneous polynomial processes. 
While in time-homogeneous polynomial processes, moments of all orders  can easily be computed by means of a matrix exponentials, this nice property does not persist in the time-inhomogeneous case. Instead, we present an approximation by means of Magnus series, which still offers a sufficient degree of tractability. 
 small application to this case.

The article is organized as follows: in Section \ref{sec2} we introduce time-inhomogeneous polynomial processes. Next, we explore in Section \ref{sec:evolution} the connection between evolution systems and transition operators from time-inhomogeneous polynomial processes and present a characterization result. In Section \ref{sec:Feller} we show that time-in\-ho\-mo\-ge\-neous polynomial processes with compact state space are Feller process. Furthermore we present the relationship between time-inhomogeneous polynomial processes and semimartingales in Section \ref{sec:semimartingale}. Some examples of time-inhomogeneous polynomial processes, as well as some counter-examples, are presented in Section \ref{sec:examples}. We present the Magnus series to compute moments of time-inhomogeneous polynomial processes in Section \ref{sec:computation}. Additionally, this section  contains some further examples of time-inhomogeneous polynomial processes whose moments can be computed using the results presented in the previous section. Some auxiliary results are relegated to  Appendix \ref{sec:appendix}.

\section{Polynomial processes}\label{sec2}
Our goal is to study time-inhomogeneous Markov processes with state space $S \subset \R^n$. The state space $S$ is only assumed to be measurable and to be sufficiently rich to uniquely identify polynomials. By  $\ccS$ we denote the Borel $\sigma$-algebra on $S$ and by $\ccS_+$ we denote the space of non-negative measurable functions $f:S \to \R$. 
For a convenient notation we fix a final time horizon $T>0$ and introduce the triangle 
\begin{equation}
 \Delta:= \left\{(s,t)|0\leq s\leq t\leq T \right\}.
 \end{equation} 
The Markov process $S$ can be characterized by a family of (time-inhomogeneous) transition operators $\left(P_{s,t}\right)_{(s,t)\in \Delta}$, given by
\begin{align}\label{eq1}
P_{s,t}f(x)=\int_{S}f(\xi)p_{s,t}(x,d\xi),  \end{align}
for all $x\in S$ and $f\in\ccS_+$. Here, $p_{s,t}(x,A)$ denotes the transition probability function, which  satisfies the following well-known properties (compare \cite{RevuzYor}, Chapter III, for further details):  
\begin{enumerate}[(i)]
\item $p_{s,t}(x,A)$ is jointly $\ccS$-measurable in $(s,t,x)$, for all  $A \in \ccS$,
\item For fixed $(s,t)\in \Delta$ and $x\in S$, $p_{s,t}(x,\cdot)$ is a probability measure  on $\ccS$,
\item For all $x\in S$ and $t\in[0,T]$, $p_{t,t}(x,\cdot)=\delta_x$, where $\delta_x$ denotes the Dirac measure
,
\item The \emph{Chapman-Kolmogorov} equations hold
$$p_{s,u}(x,A)=\int_{S}p_{s,t}(x,dy)p_{t,u}(y,A), \quad 0 \le s \le t \le u \le T.$$
\end{enumerate}
These properties immediately translate into the associated properties for the transition operators. In particular, the Chapman-Kolmogorov equations yield
$$ P_{s,u}f(x)=P_{s,t}P_{t,u}f(x), \qquad \text{for all }f \in \ccS_+. $$

It will be convenient to use the following  multi-index notation:
for $\mathbf{k}=(k_1,...,k_d)\in\mathbb{N}_0^d$  let $|\mathbf{k}|=k_1+...+k_d$ and 
$  x^{\mathbf{k}}=x_1^{k_1}\cdots x_d^{k_d}$.
By $\ccP_m(S)$ we denote the finite dimensional vector space of polynomials up to degree $m\geq 0$ on $S$, i.e. 
\begin{align} \label{def:P}
\ccP_m(S):= \{S\owns x\mapsto \sum_{|\mathbf{k}|=0}^m \alpha_{\mathbf{k}} x^{\mathbf{k}} : \alpha_{\mathbf{k}}\in\mathbb{R}\}.
\end{align}
The dimension of $\ccP_m(S)$ is denoted by $N<\infty$ and depends on $S$. Further, denote $\ccP_m:=\ccP_m(\R^d)$.

Our aim is to consider general state spaces as in  \cite{FilipovicLarsson2016} which requires some additional care as the repesentation of a polynomial as in \eqref{def:P} might not be unique. Having this in mind, we define  for  a polynomial $f \in \ccP_m$ its representations  by
$ \ccR(f)=\{ \alpha: \sum_{|\mathbf{k}|=0}^m \alpha_{\mathbf{k}} x^{\mathbf{k}} = f \} $ and
 consider the norm
\begin{align} \label{def:normm}
		\|f\|_m=\inf_{\alpha \in \ccR(f)} \max_{0\leq |\bk|\leq m}|\alpha_{\bk}|. \end{align}

Now we are in the position to introduce an appropriate class of time-dependent polynomials. Let $\tilde{S}=\Delta\times S$ be the \emph{augmented state space} (augmented by time represented by the triangle $\Delta$). By $C^1(\Delta,\mathbb{R})$ we denote the space of once continuously differentiable functions from $\Delta$ to $\R$.
 The  vector space of \emph{time-dependent polynomials} of degree at most $m$ is defined by
\begin{equation}\label{PolSpace}
\tilde{\ccP}_m(S):= \lbrace\tilde{S}\ni (s,t,x)\mapsto \sum_{|\mathbf{k}|=0}^m \alpha_{\mathbf{k}}(s,t)x^{\mathbf{k}} : \alpha_{\mathbf{k}}\in C^1(\Delta,\mathbb{R})\rbrace.
\end{equation}
In contrast to $\ccP_m(S)$, $\tilde{\ccP}_m(S)$ is an infinite dimensional vector space. In Proposition \ref{BanachPk} we show that under a suitable norm, $\tilde \ccP_m(S)$ is indeed a Banach space.

\begin{definition}
We call the transition operators $(P_{s,t})_{(s,t)\in \Delta}$ \emph{$m$-polynomial} if for all $k \in \{0,\dots,m\}$ and all $f \in \ccP_k(S)$
\begin{align*}
\tilde S \ni (s,t,x) \mapsto P_{s,t} f(x) \quad \text{is in } \tilde \ccP_k(S).
\end{align*}
If $(P_{s,t})_{(s,t)\in \Delta}$ is $m$-polynomial for all $m \ge 0$, then it is simply called \emph{polynomial}.
\end{definition}

\begin{remark}Let us emphasize some points related to this definition.
\begin{enumerate}[i)]
\item Note that we implicitly assume that moments of order up to $m$ exist, i.e. 
$$ P_{s,t} |f|(x) = \int_S |f(\xi)| p_{s,t}(x,d\xi) < \infty, $$
for every $f \in \ccP_m(S)$, $x \in S$ and $(s,t) \in \Delta$. Moreover, note that for any polynomial $f$ which vanisches on the state space $S$, $P_{s,t}f =0$, such that the transition operators are well-defined.
\item  The transition operators $(P_{s,t})_{(s,t)\in \Delta}$ are called \emph{time-homogeneous}, if 
$$ P_{s,t}=P_{0,s-t}=:P_{s-t}.$$ This is the case covered in \cite{CuchieroKellerResselTeichmann2012}, however under the weaker assumption of continuity of $s \mapsto P_s f(x)$ at $s=0$ instead of continuous differentiability which we require here. In the time-homogeneous case continuity at zero can be extended to the whole line by means of semi-group methods. The reason for this is that in the time-homogeneous case, the transition operator $P_{s,t}$ can be reduced to a Markov semi-group. In the more general case we consider here this will no longer be possible.
\item A very efficient trick for studying time-inhomogeneous Markov processes is to include time as an additional coordinate of the process. The transformed process is called \emph{space-time process}, compare \cite{Boettcher2014} for such a treatment. Applying this trick to  polynomial processes would restrict the type of time-inhomogeneity severely, as the dependence on time would be required to be of polynomial form.
\end{enumerate}
\end{remark}


In the following proposition, we show that polynomial processes can equivalently be characterized by the following two, simpler conditions. 
\begin{proposition}\label{prop:polyequi}
The transition operators $(P_{s,t})_{(s,t)\in \Delta}$ are $m$-polynomial
if and only if the following two conditions hold:
\begin{enumerate}[i)]
\item For all $k \in \{0,\dots,m\}$, all $f \in \ccP_k(S)$, and $(s,t)\in \Delta$,
\begin{align*}
S \ni x \mapsto P_{s,t} f(x) \quad \text{is in } \ccP_k(S),
\end{align*}
\item for all $k \in \{0,\dots,m\}$, all $f \in \ccP_k(S)$, and $x\in S$,
$$ \Delta \ni (s,t) \mapsto P_{s,t} f(x) \quad \text{is in }C^1(\Delta). $$
\end{enumerate}
\end{proposition}

\begin{proof}
Necessity immediately follows from the definition. For sufficiency, let $0 \leq k \leq m$ and $f \in \ccP_k(S)$. 
A subfamily of 
the mappings $S \ni x \mapsto x^{\mathbf l}, 0 \leq |\mathbf l| \leq k,$ will constitute a basis of $\ccP_k(S)$ which we denote by $ \{v_1,\ldots, v_N\}$.
Condition i) yields that   $P_{s,t}f \in \ccP_k(S)$ for every $(s,t) \in \Delta$. 
Hence there exist coefficients $\alpha_{j}^f(s,t), 0 \leq j \leq N$ such that
	\[P_{s,t}f=\sum_{j=1}^N\alpha_j^f(s,t)v_j\]
and we need to show that $\alpha_j^f\in C^1(\Delta), \ j=1,...,N$.   
 To this end, let $v_1^*,..., v_N^*$ denote the dual basis of $\left\{v_1,..., v_N\right\}$, i.e. each $v_j^*$ is an element of the dual space $(\ccP_k(S))'$  and  $v_i^*(v_j)=\ind{i=j}$. Then 
 \begin{equation*}
v_i^*(P_{s,t}f) = v_i \left(\sum_{j=1}^N\alpha_j^f(s,t)v_j\right) = \sum_{j=1}^N\alpha_j^f(s,t)v_i^*(v_j) = \alpha_i^f(s,t).   
 \end{equation*}
By condition $ii)$ the map 
$ \Delta \ni (s,t) \mapsto P_{s,t}f(x)$
is continuously differentiable for all $x \in S$. Hence, by Lemma \ref{lem:weak}, also the map
$ \Delta \ni (s,t) \mapsto P_{s,t}f \in (\ccP_k(S), \|.\|_k)$
is continuously differentiable. But the linear functionals $v_i^*$ (acting on a finite dimensional space) are automatically continuous and bounded. 
Denoting their operator norm by $\VERT v_i^*\VERT$, i.e.
$$ \VERT v_i^*\VERT = \max_{0 \neq g \in \ccP_k(S)} \frac{|v_i^*(g)|}{\|g\|_k},$$
we  obtain that
\begin{eqnarray*}
  |v_i^*(P_{s,t}f)-v_i^*(P_{s_0,t_0}f)| =   |v_i^*(P_{s,t}f-P_{s_0,t_0}f)| \leq \VERT v_i^* \VERT \cdot \|P_{s,t}f-P_{s_0,t_0}f\|_k.
\end{eqnarray*}
But here the right-hand side goes to $0$ for $(s,t) \to (s_0,t_0)$ by the continuity of $(s,t) \mapsto P_{s,t}f \in (\ccP_k(S),\|.\|_k)$ and so the continuity  of
$ \Delta \ni (s,t) \mapsto v_i^*(P_{s,t}f) = \alpha_i^f(s,t)$
follows. In  the same way we see that the $\alpha_i$'s are even continuously differentiable. 
\end{proof}

\subsection{The associated polynomial process} \label{sec:pp}
Since we are interested in Markovian processes which are semi-martingales we assume that $\Omega$ is the space of c\`adl\`ag functions $\omega:[0,T]\to S$, $X$ is the coordinate process and the filtration is the canonical filtration $(\cF_t^0)_{0 \le t \le T}$ generated by $X$. Moreover, $\ccF=\ccF_T^0$. Then, by Theorem III.1.5 in \cite{RevuzYor}, for any initial distribution and transition operators $(P_{s,t})_{(s,t)\in \Delta}$ there exists a unique probability measure on $(\Omega,\ccF)$ such that $X$ is Markovian with respect to the canonical filtration with transition operator $(P_{s,t})_{(s,t)\in \Delta}$. 

This set-up allows us to refer to a \emph{polynomial process} $X$ which is a Markov process with polynomial evolution system without imposing the Feller property (see Section \ref{sec:Feller}).  We also use the now obvious definition of $m$-polynomial processes.

\section{Evolution systems}\label{sec:evolution}
In this section we show that the transition operators from our polynomial processes in fact lead to an evolution system. Evolution systems replace semi-groups in the time-inhomogeneous case, but are far less understood. We refer to \cite{Pazy92,GulisashviliCasteren2006} for a more detailed exposition.

\begin{definition}\label{Def:evol} 
A two parameter family of bounded linear operators $U=(U_{s,t})_{(s,t) \in \Delta}$ on a Banach space $(B,\|\cdot\|)$ is called an \emph{evolution system} if the following two conditions are satisfied:

\begin{enumerate}[i)]
	\item $U_{s,s}=I$ for all $0 \leq s \leq T$ (where $I$ denotes the identity operator),
	\item $U_{r,s}U_{s,t}=U_{r,t}$, for all $0\leq r\leq s\leq t\leq T.$
\end{enumerate}
\end{definition}

The importance of evolution systems stems from the fact that,  we can associate an evolution system to a Markov process through its transition probability function. Note that a Markov semi-group $(P_t)_{0 \le t \le T}$ indeed induces an evolution system via $U(s,t)=P_{t-s}$, $(s,t) \in \Delta$.
The evolution system $(U_{s,t})_{(s,t) \in \Delta}$ is called \emph{strongly continuous} if  for each $(s,t) \in \Delta$ and $f \in B$, 
	\[\lim_{\Delta \ni (v,w)\rightarrow (s,t)}\left\|U_{v,w}f-U_{s,t}f\right\|=0.
\]

For a strongly continuous evolution system we define its \emph{infinitesimal generator} as the family $\ccG=(\ccG_s)_{0 \le s \le T}$ of linear operators with
$$\ccG_sf:=\lim_{h\downarrow 0}\frac{1}{h}(U_{s,s+h}f-f).$$
The domain 
$ \ccD(\ccG_s) $ of $\ccG_s$ is the subspace of $B$ where the above limit exists. 
The following lemma states the Kolmogorov equations for a strongly continuous evolution system.
Equation \eqref{Kolmogorovbackward} is called the \emph{Kolmogorov backward equation}, and  \eqref{Kolmogorovforward} is called the \emph{Kolmogorov forward equation}. A proof of this elementary result can be found in    \cite[Lemma 2.1]{Rueschendorf15}.

\begin{lemma} \label{KolmogorovEqs}
For the strongly continuous evolution system $U$ with infinitesimal generator $\ccG$  it holds that:
\begin{enumerate}[i)]

	\item if $f \in B$ and $s\mapsto U_{s,t}f$ is right-differentiable for $0\leq s<t$, then 
	\begin{align}\label{Kolmogorovbackward}
		\frac{d^+}{ds}U_{s,t}f=-\ccG_sU_{s,t}f.
	\end{align}
		In particular $U_{s,t}f\in\ccD(\ccG_s)$ for $0\leq s<t$.
		\item If $f\in \ccD(\ccG_s)$ for $0\leq s<t$, then 
	\begin{align}\label{Kolmogorovforward}
		\frac{d^+}{dt}U_{s,t}f=U_{s,t}\ccG_tf.
	\end{align}
		
\end{enumerate}

\end{lemma}
Let us emphasize that the equations above hold with respect to the Banach space $(B,\|.\|)$. Our next step is to show that the transition operators of an $m$-polynomial process form indeed a strongly continuous evolution system.
In the following we will call this evolution system the \emph{associated evolution system} to a Markov process. 

\begin{proposition}\label{Prop:strongevo} 
If the transition operators $(P_{s,t})_{(s,t) \in \Delta}$ are $m$-polynomial, they form a strongly continuous evolution system on $(\ccP_k(S), \|.\|_k)$ for every $0 \leq k \leq m$. 
\end{proposition}

\begin{proof}
It is straightforward to verify that $(P_{s,t})_{(s,t) \in \Delta}$ is a family of linear operators satisfying points i) and ii) of Definition \ref{Def:evol}. Furthermore they are bounded, as linear operators  on the finite-dimensional normed space $(\ccP_k(S),\|\cdot\|_k)$ are automatically bounded. 
All that remains is to show the strong continuity, i.e. that for every $f \in \ccP_k(S)$ we have
$$ \| P_{s,t}f -P_{u,v}f \|_k \to 0 \qquad \text{for} \quad \Delta \ni (s,t) \to (u,v).$$ 
By Lemma \ref{lem:weak} this is equivalent to $|P_{s,t}f(x) -P_{u,v}f(x) | \to 0$ for all $x \in S$. This, however, follows from Proposition \ref{prop:polyequi} ii).
\end{proof}

\subsection{Infinitesimal generators}
In this subsection we will show that the infinitesimal generator of a polynomial process (to be introduced below) actually coincides with the infinitesimal generator of the associated evolution system.
The \emph{infinitesimal generator}  of the transition operators $(P_{s,t})_{(s,t)\in \Delta}$ is the family of linear operators $\ccH=(\ccH_t)_{0 \le t \le T}$ satisfying
$$(\ccH_sf)(x) :=\lim_{h\downarrow 0}\frac{1}{h}(P_{s,s+h}f-f)(x), \quad x \in S.$$
Again, the domain $\ccD(\ccH_s)$ is the set
of measurable functions $f: S \to \mathbb{R}$ for which the above limit exists for all $x \in S$.

\begin{lemma}\label{Lem:Pointwise}
Let the transition operators $(P_{s,t})_{(s,t) \in \Delta}$ be $m$-polynomial with  infinitesimal generator $\ccH$. Then the following holds:
\begin{enumerate}
    \item[i)] $\ccP_m(S) \subset \ccD(\ccH_s)$ for all $0 \leq s < T.$
 	\item[ii)] If $f \in \ccP_k(S)$ for $0\leq k \leq m$ and $P_{s,t} f(x)= \sum_{|\mathbf{l}|=0}^k \alpha_{\mathbf l}(s,t) x^{\mathbf l}$, then
\begin{equation}
  \label{eq:9}
(\ccH_sf)(x)=\sum_{|\mathbf{l}|=0}^k D_2^+\alpha_{\mathbf{l}}(s,s)x^{\mathbf{l}}, \quad x \in S.  
\end{equation}
\item[iii)] $\ccH_s(\ccP _k(S)) \subset \ccP_k(S)$ for all $0 \leq k \leq m$ and all $0\leq s <T$.
\end{enumerate}
\end{lemma}
\begin{proof}
Let $0 \leq k \leq m$. Since $X$ is $m$-polynomial, $P_{.,.}f \in \tilde \ccP(S)$. For $0\leq s < T$ and $0<h\le T-s$  we obtain
\begin{align*}
\frac 1 h \left( P_{s,s+h}f-f \right)(x)
  &= \sum_{|\mathbf l|=0}^k \frac 1 h \big( \alpha_{\mathbf{l}}(s,s+h)-\alpha_{\mathbf{l}}(s,s) \big) x^{\mathbf{l}} 
  \to \sum_{|\mathbf l|=0}^k  D_2^+ \alpha_{\mathbf l}(s,s) x^{\mathbf{l}}
\end{align*}
as $h \to 0$ from above since $\alpha_{\bl}\in C_1(\Delta)$. Hence, i) and ii) follow. Part iii) immediately follows from  the representation of $\ccH$ in \eqref{eq:9}.
\end{proof}

The following proposition shows that these two concepts of infinitesimal generators actually coincide on the appropriate space of polynomials.

\begin{proposition}\label{DomainAt}
Consider $m$-polynomial transition operators whose infinitesimal generator is $(\ccH_s)_{0 \le s \le T}$ and let $(\ccG_s)_{0 \le s \le T}$ be the infinitesimal generator of the associated evolution system on $(\ccP_m(S), \|.\|_k)$. Then it holds that
 \begin{enumerate} 
      \item[i)] $\mathcal{\ccD}(\ccG_s) = \ccP_m(S)$ for every $0 \leq s < T$ and,
 	\item[ii)] if $f \in \ccP_m(S)$ then $\ccG_sf=\ccH_s f$ for every $0 \leq s < T$. 
 \end{enumerate}
 \end{proposition} 

\begin{proof}
Clearly,  $\ccD (\ccG_s) \subset \ccP_m(S)$. For the converse, consider $f\in\ccP_m(S)$, such that  $P_{s,t} f(x)$ may be represented as $\sum_{|\mathbf{l}|=0}^k \alpha_{\mathbf l}(s,t) x^{\mathbf l}$. Then  \eqref{eq:9} yields
 	\begin{eqnarray*}
 \left\| \frac 1 h \left( P_{s,s+h}f-f \right) - \ccH_s f\right\|_m 
 &=& \left\| \sum_{|\mathbf l|=0}^m \left( \frac{\alpha_{\mathbf{l}}(s,s+h)-\alpha_{\mathbf{l}}(s,s)}{h} - D_2^+ \alpha_{\mathbf l}(s,s) \right) x^{\mathbf{l}} \right\|_m\\
 &\le& \max_{0\leq \mathbf{l}\leq k}\left| \frac{\alpha_{\mathbf{l}}(s,s+h)-\alpha_{\mathbf{l}}(s,s)}{h} - D_2^+ \alpha_{|\mathbf l|}(s,s) \right|.
	 	\end{eqnarray*}
But in the last equation the right-hand side goes to $0$ for $h \searrow 0$ by the definition of $D_2^+ \alpha_{\mathbf l}(s,s)$. Hence $f \in \ccD(\ccG_s)$ and $\ccG_s f = \ccH_sf$.	
\end{proof}

\subsection{A first characterization of polynomial processes}
The successful concept of generators has one of its main applications in the following result, which allows to generate associated martingales. Recall from Section \ref{sec:pp} that we assume that $\Omega$ is the space of c\`adl\`ag functions $\omega:[0,T]\to S$, $X$ is the coordinate process and the filtration is the canonical filtration $\bbF^0:=(\ccF_t^0)_{0 \le t \le T}$ generated by $X$. Moreover, $\ccF=\ccF_T^0$. This induces a unique probability measure $\P_{0,x}$, $x \in S$ where $\P_{0,x}(X_0=x)=1$. Moreover, for any $s\in[0,T]$, we may restrict ourselves to the space of functions $\omega:[s,T]\to S$ starting at $s$ only and in a similar manner arrive at a unique probability measure $\P_{s,x}$ where $\P_{s,x}(X_s=x)=1$. The canonical filtration in this case will be denoted by $\bbF^0_s:=(\ccF_{s,t}^0)_{s \le t \le T}$.

\begin{lemma}\label{lem:martingale}
Let $X=(X_t)_{0 \leq t \leq T}$ be an $E$-valued Markov process with family of transition operators $(P_{s,t})_{(s,t) \in \Delta}$ and family of infinitesimal generators $(\ccH_s)_{0 \leq s < T}$. Suppose that for some $f \in \cap_{0 \leq u < T}\ccD(\ccH_u)$ and every $s \in [0,T)$ the following holds:  
 
\begin{enumerate}[a)]
    \item The random variables
\begin{equation}
  \label{eq:31}
 M_{s,t}^f:=f(X_t) - f(X_s) - \int_{s}^t\ccH_uf(X_u)du, \quad  s \leq t \leq T
\end{equation}
are well-defined,
\item $\int_s^T P_{s,u} |\ccH_uf| (x) du < \infty$ for all $0 \leq s \leq T$ and all $x \in E$.
\end{enumerate}
Then the following are equivalent: 
\begin{enumerate}[i)]
    \item $(M_{s,t}^f)_{s \leq t \leq T}$ is a $(\mathbb{P}_{s,x}, \mathcal F_{s,t})$-martingale for every $x \in E$ and every $s \in [0,T)$.
    \item For all $(s,t) \in \Delta$ and every $x \in E$ we have 
\begin{align}\label{lem3.5ii}
P_{s,t}f(x) = f(x) + \int_{s}^{t}P_{s,u}\ccH_uf(x)du. \end{align}
\end{enumerate} 
\end{lemma}
\begin{proof}
$i) \Rightarrow ii)$: the martingale property yields for all $0 \leq s \leq v \le t \leq T$, that
\begin{equation}
  \label{eq:25}
\mathbb E_{s,x}[M_{s,v}^f]= \mathbb E_{s,x}[M_{s,t}^f].  
\end{equation}
 Under condition b) we obtain, using the Fubini theorem,
\begin{eqnarray*}
\mathbb E_{s,x}[M_{s,v}^f] &=& \mathbb E_{s,x} [f(X_v)] - \mathbb E_{s,x}[f(X_s)] - \mathbb E_{s,x} \left[\int_{s}^v\ccH_uf(X_u)du \right] \\
&=& P_{s,v}f(x) - f(x) - \int_{s}^{v}P_{s,u}\ccH_uf(x)du.  
\end{eqnarray*}
Choosing $v=s$ in (\ref{eq:25}) we obtain $ii)$.

\noindent $ii) \Rightarrow i)$: Let $0 \leq s_1 \leq s_2 < t \leq T$. By the Markov property and a further application of the Fubini theorem, 
\begin{align*}
  \mathbb{E}_{s_1,x}[ M_{s_1,t}^f | \ccF_{s_1,s_2}] -M_{s_1,s_2}^f &=   \mathbb{E}_{s_1,x}[ M_{s_1,t}^f-M_{s_1,s_2}^f | \ccF_{s_1,s_2}] \\
&= \mathbb{E}_{s_1,x}[ f(X_{t}) - f(X_{s_2}) - \int_{s_2}^{t}\ccH_uf(X_u)du | \ccF_{s_1,s_2}] \\
&= \mathbb{E}_{s_2,X_{s_2}}[f(X_{t}) - f(X_{s_2}) - \int_{s_2}^{t}\ccH_uf(X_u)du] \\
&=  P_{s_2,t}f(X_{s_2}) - f(X_{s_2}) - \int_{s_2}^{t}P_{s_2,u}\ccH_uf(X_{s_2})du \overset{ii)}{=} 0.\qedhere
\end{align*}
\end{proof}

The martingale property of $M^f$ in the above lemma can be deduced from the polynomial property if $m$ is even (because then $|\cdot|^m$ is  polynomial). This is already the case for time-homogeneous polynomial processes, see Theorem 2.10 and Remark 2.11 in \cite{CuchieroKellerResselTeichmann2012}.

\begin{proposition}\label{prop:polymartingale}
Let $m \ge 2$ be even and let  $X=(X_t)_{0 \leq t \leq T}$ be an $S$-valued $m$-polynomial process and  $(\ccH_s)_{0 \leq s < T}$ be its family of infinitesimal generators. For  $f \in \ccP_m(S)$ and $s \in [0,T]$, 
$$ M_{s,t}^f:=f(X_t) - f(x) - \int_{s}^t\ccH_uf(X_u)du, \quad s \le t \le T$$
is well-defined and a $(\mathbb{P}_{s,x}, \bbF^0_{s})$-martingale for every $x\in S$.
\end{proposition}
\begin{proof}
By Lemma \ref{Lem:Pointwise} we have $\ccP_m(S) \ni f  \subset \ccD(\ccH_u)$ for all $0 \leq u < T$ and
\begin{align} \ccH_uf(X_u) = \sum_{|\mathbf l|=0}^m D_2^+ \alpha_{\mathbf l}^f(u,u) (X_u)^{\mathbf l} \label{temp394} \end{align}
with  $\alpha_{\mathbf l}^f \in C^1(\Delta)$, such that $\sup_{u \in [s,t]} |\alpha_{\mathbf l}^f(u,u)|< \infty$. Moreover, $u \mapsto X_u(\omega)$ is c\`adl\`ag for all $\omega$, such that   $u \mapsto \ccH_uf(X_u)$ is indeed integrable on $(s,t)$ and hence $M_{s,t}^f$ is well-defined. Moreover, for all $0 \leq s \leq u < T$ the function $\ccH_u f$ is in $\ccP_m(S)$ and hence $P_{s,u}|\ccH_uf|(x)< \infty$ for all $x \in S$ (as is required for an $m$-polynomial process). Hence condition  a) of Lemma \ref{lem:martingale} is satisfied. Towards condition b) note that Equation \eqref{temp394}  yields
\begin{align*}
   P_{s,u} |\ccH_uf(x)| \le \sum_{|\mathbf l|=0}^m |D_2^+ \alpha_{\mathbf l}^f(u,u)| \cdot P_{s,u} |e_{\bl}| (x) 
\end{align*}
with $e_{\bl}=x^\bl$. Moreover, $|e_\bl(x)| \le 1+g(x)$ with $g(x)=|x|^m$. Since $m$ is even, $g$ is polynomial. Then, Propsition \ref{prop:polyequi} yields that $P_{s,u} g(x)$ depends continuously on $u$ for all $x \in S$, and condition b) of Lemma \ref{lem:martingale} is satisfied as well.

Moreover, for all $0 \leq s \leq t \leq T$ and $x \in S$ from the Kolmogorov forward equation \eqref{Kolmogorovforward} we obtain that
\begin{eqnarray*}
P_{s,t}f(x) - f(x)  - \int_{s}^{t}P_{s,u}\ccH_uf(x)du 
= P_{s,t}f(x) - f(x) - \int_{s}^{t} \frac{d^+}{du} P_{s,u}f(x)du =0.
\end{eqnarray*}
Now the Proposition follows from the equivalence $i) \Leftrightarrow ii)$ in Lemma \ref{lem:martingale}. 
\end{proof}

The following theorem is the main result of this section. It extends  Theorem 2.10 in \cite{CuchieroKellerResselTeichmann2012} to the  time-inhomogeneous case. 
 As above we denote
$$M_{s,t}^f:=f(X_t) - f(X_s) - \int_{s}^t\ccH_uf(X_u)du, \quad  (s,t) \in \Delta.$$
\begin{theorem}\label{thm:mainequivalence}
Let $X=(X_t)_{0 \leq t \leq T}$ be an $S$-valued Markov process with family of transition operators $(P_{s,t})_{(s,t) \in \Delta}$ and family of infinitesimal generators $(\ccH_s)_{0 \leq s < T}$. Moreover let $m \ge 2$ be an even number. 
Then the following are equivalent: 
\begin{enumerate}[i)]
    \item $X$ is $m$-polynomial. 
    \item The following conditions  hold:
  \begin{enumerate}[a)]
    \item $\int_s^T P_{s,u}|f|(x) du< \infty$ for all $x \in S, \ 0 \le s \le T $ and $f \in \ccP_m(S)$.
    \item $(M_{s,t}^f)_{s \leq t < T}$ is well-defined and a $(\mathbb{P}_{s,x}, \bbF^0_s)$-martingale for every $x \in S, f \in \ccP_m(S)$ and all $s \in [0,T)$.
    \item  $\ccH_s (\ccP_k(S)) \subset \ccP_k(S)$ for all $0 \leq k \leq m$ and $0 \leq s < T$.
      \item For all $0 \leq k \leq m$ and $f \in \ccP_k(S)$ there exist $b_{\mathbf l}^f \in C([0,T)), 0 \leq |\mathbf l| \leq k,$ such that
    \begin{equation}
      \label{eq:15}
      \ccH_sf(x)= \sum_{|\mathbf l|=0}^k b_{\mathbf l}^f(s) x^{\mathbf l}
    \end{equation}
for all $x \in S$ and $s \in [0,T)$.  
  \end{enumerate}
\end{enumerate}
\end{theorem} 

\begin{proof}
To begin with, we note that the implication $i) \Rightarrow ii)$ readily follows from Lemma \ref{Lem:Pointwise}, Proposition \ref{prop:polymartingale} and  the definition of an $m$-polynomial process. 

For the reverse direction 
$ii) \Rightarrow i)$,  let $0 \leq k \leq m$ and $f \in \ccP_k(S)$. We have to show that $P_{s,t}f \in \ccP_k(S)$ for all $(s,t) \in \Delta$ and that $\Delta \ni (s,t) \mapsto P_{s,t}f \in \mathcal (\ccP_k(S),\|.\|_k)$ is continuously differentiable. 

First, note that by condition c), $f \in \ccD (\ccH_s)$ for all $0 \leq s < T$ and $\ccH_s f \in P_k(S)$. 
Then, conditions a) and b) guarantee the applicability of Lemma \ref{lem:martingale} and we obtain Equation \eqref{lem3.5ii}.
Note that this already implies that the map $[s,T) \ni t \mapsto P_{s,t}f(x)$ is continuous. 

Second, from assumption d) we obtain, by applying $P_{s,u}$ to Equation \eqref{eq:15}, that
$$ P_{s,u} \ccH_uf(x) = \sum_{|\mathbf l|=0}^k b_{\mathbf l}^f(u) [P_{s,u}e_{\mathbf l}](x),$$
where again $e_{\mathbf l}(x)= x^{\mathbf l}$. Since, $u \mapsto b_{\mathbf l}^f(u)$ is continuous and, as already shown, $u \mapsto P_{s,u}e_{\mathbf l}(x)$ is continuous, we obtain that $u \mapsto P_{s,u} \ccH_uf(x)$ is continuous. Since $X$ is polynomial, the transition operators $(P_{s,t})$ are a strongly continuous evolution system by Proposition \ref{Prop:strongevo} and Lemma \ref{KolmogorovEqs}  yields that the Kolmogorov forward equation 
  \begin{equation}
    \label{eq:30}
    \frac{d^+}{dt} P_{s,t}f(x) = P_{s,t} \ccH_t f(x).
  \end{equation}   
 holds. The third and final step will be to use this property together with Lemma \ref{lem:pazy} to obtain that  $\Delta \ni (s,t) \mapsto P_{s,t}f \in \mathcal (\ccP_k(S),\|.\|_k)$ is continuously differentiable. 

In order to apply Lemma \ref{lem:pazy} we represent the Kolmogorov equation with respect to a basis in the space of polynomials: let $v_1, \ldots, v_N$ denote a basis of $\ccP_k(S)$ and consider $s \in [0,T)$. Recall that $\ccH_s : \ccP_k(S) \to \ccP_k(S)$ such that we may choose a representing matrix of $\ccH_s$, which we denote by  $\mathbf A_s \in \R^{N \times N}$ (see Section \ref{sec:matrices} for details, in particular we have $\mathbf A_s = U \ccH_s U^{-1}$ with  the linear map $U : \ccP_k(S) \to \R^N$  being defined by $Uv_j=e_j, j=1,\ldots,N$, with the standard basis of $\R^N$, $\{e_1,\dots,e_N\}$).

As $(\mathbf A_s)_{0 \le s <T}$ constitutes a continuous family of matrices, Lemma \ref{lem:pazy} yields a unique solution to the initial value problem 
$$ \left\{
  \begin{array}{cl}
\frac{d}{dt} \mathbf V(s,t)= \mathbf V(s,t) \mathbf A_t, & s < t <T \\
 \mathbf V(s,s)= \mathbf I. 
  \end{array}\right.$$
Define $V_{s,t}:= U^{-1}\mathbf V(s,t) U$. Then  $V_{s,t} : \ccP_k(S) \to \ccP_k(S)$, for every $f \in \ccP_k(S)$ the map
$$ \Delta \ni (s,t) \mapsto V_{s,t}f \in (\ccP_k(S),\|.\|_k)$$
is continuously differentiable, $V_{s,s}=I$ (the identity operator on $\ccP_k(S)$) and
\begin{equation}
  \label{eq:33}
  \frac d {dt} V_{s,t} = - \ccH_s V_{s,t}.
\end{equation}

Finally, we show that $P_{s,t}f=V_{s,t}f$ for all $f \in \ccP_k(S)$ which finishes the proof: fix $x \in S$ and $0 \le s < t \le T$ and define $$W(r):=P_{s,r}V_{r,t}f(x), \quad s \leq r \leq t.$$ 
Then the function $W$ is right-differentiable with
\begin{eqnarray*}
\frac{d^+}{dr} W(r) &=& \left( \frac{d^+}{dr} P_{s,r} \right) V_{r,t} f(x) + P_{s,r} \left( \frac{d^+}{dr} V_{r,t} \right) f(x) \\
&=& P_{s,r} \ccH_r V_{r,t}f(x) - P_{s,r} \ccH_r V_{r,t}f(x) = 0.  
\end{eqnarray*}
This shows that $W$ is constant. In particular $W(s)=W(t)$ which is equivalent to $V_{s,t}f(x)=P_{s,t}f(x)$ and the proof is finished.
\end{proof}

\section{The Feller property}\label{sec:Feller}
In this section we establish a sufficient criterion for uniqueness of the associated transition operators - the \emph{Feller} property. One central result will be that polynomial processes  with compact state space are  Feller processes. We begin with a short introduction to the topic following \cite{RevuzYor,Boettcher2014}. To this end we denote for a set $S \subset \R^d$ by
\begin{equation*}
  C_0(S):=\{ f \in C(S) :  \forall \varepsilon > 0 \; \exists K \subset S \text{ compact such that } |f(x)| < \varepsilon \text{ for } x \notin K\}
\end{equation*}
the space of continuous functions on $S$  vanishing at infinity. Recall that equipped with the norm
$ \|f\|_\infty:= \sup_{x \in S} |f(x)| $, $C_0(S)$ is a Banach space. 

\begin{definition} \label{defFeller}
The transition operators $(P_{s,t})_{(s,t) \in \Delta}$ are called \emph{Feller},  if  for all $f \in C_0(S)$:
\begin{enumerate}[i)]
\item $P_{s,t}f \in C_0(S)$,
\item $\left\|P_{s,t}f\right\|_{\infty}\leq \left\|f\right\|_{\infty}$ for all $(s,t) \in \Delta$,
\item $(P_{s,t})_{(s,t) \in \Delta}$ is strongly continuous on the Banach space $(C_0(S),\|.\|_\infty)$.
\end{enumerate}
In this case the family $(P_{s,t})_{(s,t) \in \Delta}$ is also called a \emph{Feller evolution system}.
\end{definition}

Note that for any Markov process the associated transition operators $P_{s,t}$ are defined on $C_0(S)$ since functions in $C_0(S)$ are bounded. Moreover, also the property $ii)$ is satisfied for all Markov processes. The relevance for Feller processes stems from the fact that to a Feller process one can associate a unique evolution system, see Proposition III.2.2. in \cite{RevuzYor}.

\begin{proposition}\label{prop:Feller}
Let $S \subset \R^d$ be compact. Then an $S$-valued polynomial process $X=(X_t)_{0\leq t \leq T}$ is a Feller process. 
\end{proposition}
\begin{proof} 
Denote by $(P_{s,t})_{(s,t) \in \Delta}$ transition operators of $X$. Since  $S \subset \R^d$ is compact we have $C_0(S)=C(S).$ We verify the properties of Definition \ref{defFeller}. Clearly ii) holds. Next, we show i): let $f \in C(S)$, $x \in S$ and $\epsilon >0$. By the Stone-Weierstra{\ss} theorem, we know that there exists $k\in\mathbb{N}$ and $p\in\ccP_k(S)$ such that $\left\|f - p\right\|_{\infty}<\epsilon/3$. Then, for $y \in S$ we have that
\begin{eqnarray*}
  |P_{s,t}f(x)-P_{s,t}f(y)| &\leq& |P_{s,t}(f-p)(x)| + |P_{s,t}p(x)-P_{s,t}p(y)| + |P_{s,t}(p-f)(y)| \\
&\leq& \|P_{s,t}(f-p)\|_\infty + |P_{s,t}p(x)-P_{s,t}p(y)| + \|P_{s,t}(p-f)\|_\infty \\
&\leq& \|f-p\|_\infty + |P_{s,t}p(x)-P_{s,t}p(y)| + \|f-p\|_\infty \\
&\leq& 2 \epsilon/3 + |P_{s,t}p(x)-P_{s,t}p(y)|.
\end{eqnarray*}
Since $x\mapsto P_{s,t}p(x)$ is polynomial and therefore continuous, there is $\delta > 0$ such that with $|x-y| < \delta$ it holds that $|P_{s,t}p(x)-P_{s,t}p(y)|< \epsilon/3$ and i) follows.

In view of  iii) we have to verify that for every $f\in C(S)$ and $(v,w)\in \Delta$ the following holds: for all $\epsilon>0$ there exists $\delta>0 $ such that  for $(s,t) \in \Delta$ with $\left|(s,t)-(v,w)\right|<\delta $ we have $ \left\|P_{s,t}f-P_{v,w}f\right\|_{\infty}<\epsilon$. To this end, consider $f\in C(S)$ and $\epsilon >0$. As above there exists $k\in\mathbb{N}$ and $p\in\ccP_k(S)$ such that $\left\|f - p\right\|_{\infty}<\epsilon/3$.
Then, 
\begin{eqnarray*}
\left\|P_{s,t}f-P_{v,w}f\right\|_{\infty}&\leq& 
\left\|P_{s,t}(f-p)\right\|_{\infty} + \left\|P_{s,t}p-P_{v,w}p\right\|_{\infty} + \left\|P_{v,w}(p-f)\right\|_{\infty}\\
&\leq& \left\|f-p\right\|_{\infty} + \left\|P_{s,t}p-P_{v,w}p\right\|_{\infty} + \left\|p-f\right\|_{\infty}\\
&\leq& 2\epsilon/3 + \left\|P_{s,t}p-P_{v,w}p\right\|_{\infty}.
\end{eqnarray*}

All that remains is to choose $(s,t)$ suitably such that the second term is smaller than $\epsilon/3$. 
On the finite dimensional Banach space $\ccP_k(S)$ all norms are equivalent. Hence, there exists a constant $c_0>0$ such that
$ \left\|P_{s,t}p-P_{v,w}p\right\|_{\infty}\leq c_0 \left\|P_{s,t}p-P_{v,w}p\right\|_{k}$
with  $\|f\|_m=\max_{0\leq |\bk|\leq m}|\alpha_{\bk}|$ . 
In Proposition \ref{Prop:strongevo} we showed that $(P_{s,t})$ is strongly continuous on $\ccP_k(S)$. As a consequence, there exists $\delta>0$ such that for $\left|(s,t)-(v,w)\right|<\delta$ we have  $\left\|P_{s,t}p-P_{v,w}p\right\|_{k}<\epsilon/(3c_0)$. So, for this choice of $(s,t)$ we obtain 
$$\left\|P_{s,t}p-P_{v,w}p\right\|_{\infty}\leq c_0 \left\|P_{s,t}p-P_{v,w}p\right\|_{k}<\epsilon/3$$
and iii) follows.
\end{proof}

\begin{remark}\label{rem:Feller}
\begin{enumerate}[(i)]
\item Even if we know that affine processes on the canonical state space have the Feller property (see \cite{Filipovic05}), this may fail to hold on a more general state space as considered here. 
\item The proof of Proposition \ref{prop:Feller} actually shows that for  polynomial transition operators properties ii) and iii) of Definition \ref{defFeller} hold (on a possibly not compact state space). Moreover, instead of property i) only the weaker asseration $P_{s,t} C(S) \subset C(S)$ holds, which, however, will be sufficient to follow the usual augmenation of canonical filtrations in the beginning of the following section.
\end{enumerate}
\end{remark}

\section{Semimartingale characteristics}
\label{sec:semimartingale}

We begin our study of polynomial processes which are semimartingales. 
For general facts on semimartingales and stochastic analysis used in this article we refer to \cite{JacodShiryaev}.
Recall from Section \ref{sec:pp} that we assumed $\Omega$ to be the space of c\`adl\`ag functions $\omega:[0,T]\to S$, $X$ is the coordinate process and the filtration is the canonical filtration $\bbF^0:=(\ccF_t^0)_{0 \le t \le T}$ generated by $X$. Moreover, $\ccF^0=\ccF_T^0$. This induces a unique probability measure $\P_{0,\nu}$, $x \in S$ where $\nu$ is the initial distribution of $X$, i.e.\ $\P_{0,x}(X_0\in A)=\nu(A)$.

However, this filtration is neither complete nor right-continuous, which is  essential for the application of semi-martingale calculus. We follow \cite{RevuzYor}, Section III.2 for the usual augmentation for Markov-processes: consider an initial distribution $\nu$. Then, denote by $\ccF^\nu$ the completion of $\ccF$ with respect to $\P_{0,\nu}$ and by $\bbF^\nu:=(\ccF_t^\nu)_{0 \le t \le T}$ the filtration obtained by adding all the $\P_{0,\nu}$-negligible sets in $\ccF^\nu$ to each $\ccF_t^0$, $t \in [0,T]$. Finally, we set 
\begin{align*}
\ccF_t = \bigcap_\nu \ccF_t^\nu, \qquad \ccF := \bigcap_\nu \ccF^\nu
\end{align*}
and denote $\bbF:=(\ccF_t)_{0 \le t \le T}$. By Proposition III.2.10 in \cite{RevuzYor}, $\bbF$ satisfies the usual conditions if $X$ is Feller which fails in general for the polynomial processes considered here.
However, as noted in Remark \ref{rem:Feller} ii), a weaker property holds, namely ii) and iii) of Definition \ref{defFeller} are satisfied and instead of i) it holds that $P_{s,t} C(S) \subset C(S)$. In particular, this property holds for bounded continuous functions. An inspection of the proof of Proposition III.2.10 reveals that this is actually sufficient for right-continuity of $\bbF$. In an analogous way we define the filtration $\bbF_s=(\ccF_{s,t})_{s \le t \le T}$ starting at $s \in [0,T]$.

Finally, analogous to Section \ref{sec:pp} we define the respective quantities starting from time point $s\in[0,T]$, i.e.\ $\P_{s,\nu}$, $\bbF^\nu_s$ and $\bbF_s$.
Now we are ready to recapitulate standard notions from semi-martingale calculus. For the c\`adl\`ag process $X$ we define $X_-$ and $\Delta X$ by 
\[\begin{cases}X_{0-} &= X_0, \quad  X_{t-} = \lim_{s \uparrow t} X_s \quad \text{ for }t>0,\\
\Delta X_t &= X_t - X_{t-}  . \end{cases}
\]

The process $X$ is a  $(\P,\bbF)$-\emph{semimartingale} if it has a decomposition $X=X_0+N+M$ where $X_0$ is $\cF_0$-measurable, $N$ is c\`adl\`ag, $\bbF$-adapted, has paths of finite variation over each finite interval with $N_0=0$ and $M$ is a $(\P,\bbF)$-local martingale starting in $0$. We emphasized the obvious dependence on $\P$ and $\bbF$ because in the following we will refer to different filtrations and measures. For the following definitions we keep this dependence in mind, but to facilitate notation, we will typically drop it in our notation.


We can associate an integer-valued random measure $\mu^X$ with the jumps of $X$ by 
	\begin{align} \label{def:muX}
		\mu^X(dt,dx) = \sum_{s \ge 0} \ind{\Delta X_s \neq 0} \delta_{(s,\Delta X_s)}(dt,dx); 
	\end{align}
here $\delta_a$ is the Dirac measure at point $a$. We denote the compensator, or the dual predictable projection, of the random measure $\mu^X$ by $\nu$. This is the unique $\bbF$-predictable random measure which renders stochastic integrals with respect to $\mu^X-\nu$  local martingales. 

The semimartingale $X$ is called \emph{special} if $N$ is predictable. In this case, the decomposition  $X=X_0+N+M$ is unique, and we call it the \emph{canonical decomposition}. The local martingale part $M$ can be decomposed in a continuous local martingale part, which we denote by $X^c$, and a purely discontinuous local martingale part, $X-X^c$.

A for us essential concept is the \emph{characteristics} of a semimartingale. The characteristics of a general semimartingale typically involve a trunction function $h$. Here, however, we will only deal with special semi-martingales, such that this is not necessary and one can choose $h(x)=x$.  
The \emph{characteristics} of the special semimartingale $X$ with unique decomposition $X=X_0+B+M $ is the triplet $(B,C,\nu)$ where $B=(B^i)$ is a process of finite variation,  $C=(C^{ij})$ with $C^{ij}=\scal{X^{i,c}}{X^{j,c}}$ and $\nu=\nu^X$ is the compensator of $\mu^X$ defined in Equation \eqref{def:muX}.

\begin{lemma}
Let $X=(X_t)_{0 \leq t \leq T}$ be an $S$-valued $m$-polynomial process with $m \ge 2$ and $S\subset \R^d$ being closed. Then for all $ 0 \leq s \leq T$ and all $x \in S$ the process $(X_t)_{s \leq t \leq T}$ is a special semi-martingale with respect to the stochastic basis $(\Omega,\ccF,(\ccF_{s,t})_{s \leq t \leq T},\mathbb{P}_{s,x})$. 
\end{lemma}

\begin{proof}
We denote the infinitesimal generators of $X$ by $(\ccH_s)_{0 \leq s < T}$.
By assumption,  $X$ is $2$-polynomial and we therefore may apply  Proposition \ref{prop:polymartingale} to the projection to the $i$-th coordinate,  $\pi_i(x_1,\ldots,x_d):= x_i,$ $1 \leq i \leq d$. Hence, 
\begin{equation}
  \label{eq:37}
M_{s,t}^{i}:=M_{s,t}^{\pi_i}=X_t^{i}- X_s^{i}- \int_s^t \ccH_u \pi_i(X_u)du, \qquad s \le t \le T
\end{equation} 
is a $(\mathbb{P}_{s,x}, \bbF^0_{s})$-martingale. By Proposition 2.2 in \cite{NeufeldNutz:Measurability} it is moreover a $(\mathbb{P}_{s,x}, \bbF_s)$-martingale.
Solving for $X_t^{i}$ we see that $X^i$ is the sum of an absolutely continuous (and therefore predictable) process and a martingale, hence a special semimartingale and so is $X$.
\end{proof}

\begin{theorem} \label{PolProcChar} 
Fix $ 0 \leq s \leq T$ and  $x \in S$ and 
denote the semimartingale characteristics of the $m$-polynomial process $(X_{s,t})_{s \le t \le T}$ with respect to $(\ccF_{s,t})_{s \leq t \leq T}$ and $\mathbb{P}_{s,x}$ by $(B_t,C_t,\nu_t)_{s \le t \le T}$. If $m \ge 2$,  then the following holds: 
\begin{enumerate}[i)]
\item There exist measurable functions $b^{i},a^{ij} : [0,T) \times S \to \R, 1 \leq i,j \leq d$, not depending on $s$ and $x$, such that for all $0 \leq t < T$ and $\xi \in S$
$$ b^i(t,.) \in \ccP_1(S), \  a^{ij}(t,.)\in\ccP_2(S), $$ 
 with $ b^{i}(.,\xi), a^{ij}(.,\xi) \in C[0,T)$  and 
\begin{eqnarray} 
\label{c1} B_{t}^{i}&=&\int_s^t b^i(u,X_u)du,\\
\label{c2} C_{t}^{ij}&+&\int_s^t\int_{\mathbb{R}^d}\xi_i\xi_j\nu(du,d\xi)=\int_s^t a^{ij}(u,X_u)du, 
\end{eqnarray}
for all $0 \leq s \leq t \leq T$ and $1 \leq i,j \leq d$.

\item There exists a measurable function $c=(c^{ij})_{i,j=1}^d : [0,T) \times S \to \R^{d \times d}$, taking values in the set of positive semi-definite matrices, such that 
\begin{equation}
\label{c3} C_{s,t}^{ij}=\int_s^t c^{ij}(u, X_u)du, \qquad 0 \leq s \leq t \leq T.
\end{equation} 

\item For each $0 \leq t \leq T$ there exists a positive transition kernel  $K_t$ from $(S,\mathcal{S})$ into $(\mathbb{R}^d,\ccB(\mathbb{R}^d))$, which  integrates $\left(|x|^2\wedge 1\right)$ and satisfies $K_u(x,\{0\})=0$, such that 
\begin{equation}
\label{c3_2} \nu(\omega; dt, d\xi)=K_t(X_t(\omega),d\xi)dt. 
\end{equation}  
Moreover, for all $3 \leq | \mathbf{k}| \leq m$ there exist $\alpha_{\mathbf l} \in C[0,T), 0 \leq |\mathbf l| \leq |\mathbf k|$ such that  
\begin{equation} 
\int_{\mathbb{R}^d}\xi^{\mathbf{k}}K_t(x,d\xi)=\sum_{\vert \mathbf{l}\vert =0}^{\vert \mathbf{k}\vert} \alpha_{\mathbf{l}}(t)x^{\mathbf{l}}, \quad x \in S, t \in [0,T). \label{c4} 
\end{equation}  
\end{enumerate}
\end{theorem}


\begin{proof} 
Our proof follows the proof of the first implication in Proposition 2.12 in \cite{CuchieroKellerResselTeichmann2012} additionally taking care on the time-inhomogenity.

i) Equation \eqref{eq:37} implies that  
$   B_{t}^{i} = \int_s^t \ccH_u\pi_i(X_u)du$ and we set  
\begin{equation}
  \label{eq:bi}
b^i(u,\xi):= \ccH_u \pi_i(\xi), \quad 0 \leq u \le T,\  \xi \in S.  
\end{equation}
Since $\pi_i \in \ccP_1(S)$ and, by Theorem \ref{thm:mainequivalence} c), $\ccH_u \pi_i \in  \ccP_1$, it follows that $b^i(u,.) \in \ccP_1(S)$. Moreover, by Theorem \ref{thm:mainequivalence} d)  the map $u \mapsto \ccH_u \pi_i(\xi)=b^{i}(u,\xi)$ is continuous. 

Aplying It\^o's formula with $f_\bk(x)=x^\bk$ implies
\begin{eqnarray}
	\nonumber f_{\mathbf k}(X_t)&=& f_{\mathbf k}(x) + \int_s^t \sum_{i=1}^d D_i f_{\mathbf k}(X_{u-}) dX_{u}^{i}	 +  \frac{1}{2}\int_s^t \sum_{i,j=1}^d D_{ij} f_{\mathbf k}(X_{u-}) dC_{s,u}^{ij}\\
	 & +& \sum_{u\leq t}\left(f_{\mathbf k}(X_{u})- f_{\mathbf k}(X_{u-})- \sum_{i=1}^d D_i f_{\mathbf k}(X_{u-})\Delta X_u^i\right). \label{temp668}
\end{eqnarray}
First, we concentrate on the jump part. Denote by $\mathbf e_i$ the $d-$dimensional vector whose $i-$th entry is 1 and the rest of the components are 0 and by $\mu^X$ the random measure associated to the jumps of $X$. Since $D_i x^\bk \xi = k_i x^{\bk - \be_i} $, 
we obtain
\begin{align}
\lefteqn{
    (X_{u-}+\xi)^{\mathbf k}- (X_{u-})^{\mathbf k}- \sum_{i=1}^d D_i k_i(X_{u-})^{\mathbf k-\mathbf e_i}\xi^i} \qquad \notag\\
 &=  \sum_{|\mathbf{l}|=0}^{|\mathbf k|} \binom{\mathbf k}{\mathbf l} X_{u-}^{\mathbf k- \mathbf l}\xi^{\mathbf l}- (X_{u-})^{\mathbf k}- \sum_{i=1}^d  k_i(X_{u-})^{\mathbf k-\mathbf e_i}\xi^{\mathbf e_i}\notag\\
 & =     \sum_{|\mathbf{l}|=2}^{|\mathbf k|} \binom{\mathbf k}{\mathbf l} X_{u-}^{\mathbf k- \mathbf l}\xi^{\mathbf l} =: W(u,\xi). \label{eq676}
\end{align}
Recall from Equation \eqref{eq:37} that $X^{i}_t=X^i_s + M_{s,t}^{i} + B_{s,t}^{i}$, such that 
\begin{eqnarray}\label{ItoFormula}
	\nonumber f_{\mathbf k}(X_t)&=& f_{\mathbf k}(x) + \int_s^t \sum_{i=1}^d D_i f_{\mathbf k}(X_{u-}) dM_{s,u}^{i} + \int_s^t \sum_{i=1}^d  D_i f_{\mathbf k}(X_{u-}) dB_{s,u}^{i}\\ 
	& +&  \frac{1}{2}\int_s^t \sum_{i,j=1}^d D_{ij} f_{\mathbf k}(X_{u-}) dC_{s,u}^{ij} + \int_s^t \int_{\mathbb{R}^d}W(u,\xi) \mu^X(du,d\xi).
\end{eqnarray}
Compensating $\mu^X$ we obtain that $f_\bk(X)$ is a special semi-martingale with unique semi-martingale decomposition. With the notation from  Equation \eqref{eq:31} we denote the  local martingale part by $M_{s,t}^{f_{\mathbf k}}$ and obtain that 
\begin{eqnarray}
f_{\mathbf k}(X_t) &=& f_{\mathbf k}(x) + M_{s,t}^{f_{\mathbf k}} + \int_s^t \sum_{i=1}^d  D_i f_{\mathbf k}(X_{u-}) dB_{s,u}^{i} + 
 \frac{1}{2}\int_s^t \sum_{i,j=1}^d D_{ij} f_{\mathbf k}(X_{u-}) dC_{s,u}^{ij} \notag\\
&&+ \int_s^t \int_{\mathbb{R}^d}W(u,\xi) \nu(du,d\xi) \notag\\
&=& f_{\mathbf k}(x) + M_{s,t}^{f_{\mathbf k}} + \int_s^t \ccH_uf_{\mathbf k}(X_u)du;
\label{temp689}
\end{eqnarray}
the last equality follows from Lemma \ref{lem:martingale}, in particular from Equation \eqref{eq:31}. 
We apply this representation to the quadratic polynomials $f_{ij}(x):=x_ix_j$, $1 \leq i,j \leq d$ and obtain
\begin{align}\label{Genf2}
 \lefteqn{ \int_s^t \ccH_uf_{ij}(X_u)du } \qquad\\
&=  \int_s^t \sum_{k=1}^d  D_k f_{ij}(X_{u-}) dB_{s,u}^{k} +  \frac{1}{2}\int_s^t \sum_{k,l=1}^d D_{kl} f_{ij}(X_{u-}) dC_{s,u}^{kl}
+  \int_s^t \int_{\mathbb{R}^d}\xi_i \xi_j \nu(du,d\xi) \nonumber \\  
\nonumber &= \int_s^t  X_{u-}^{j}  dB_{s,u}^{i} + \int_s^t  X_{u-}^{i}  dB_{s,u}^{j} +  C_{s,t}^{ij} + \int_s^t \int_{\mathbb{R}^d}\xi_i \xi_j \nu(du,d\xi).\\ 
\nonumber &= \int_s^t  X_{u}^{j} b^{i}(u,X_u)du + \int_s^t   X_{u}^{i} b^{j}(u,X_u) du +  C_{s,t}^{ij} + \int_s^t \int_{\mathbb{R}^d}\xi_i \xi_j \nu(du,d\xi)
\end{align}
with $b^{i}(.,.)$ from Equation  \eqref{eq:bi}. In view of our claim, set 
$$a^{ij}(u,\xi):= \ccH_u f_{ij}(\xi) - \xi_j b^{i}(u, \xi) - \xi_i b^{j}(u, \xi).$$
Since $\ccH_u f_{ij}\in \ccP_2(S)$, by Theorem  \ref{thm:mainequivalence} c) and d), we obtain that  $a^{ij}(u,.) \in \ccP_2(S)$ and $ a^{ij}(., \xi) \in C[0,T)$.  
Moreover, representation \eqref{c2} follows now from  (\ref{Genf2}) and the proof of i) is finshed.
\\

ii): To begin with, define the predictable and increasing process $A$ by
$$A_{s,t}(\omega):=\int_s^t\int_{\mathbb{R}^d}|\xi|^2\nu(\omega; du,d\xi), \quad s \le t \le T.$$
Then, by the same arguments as in the proof of Theorem II.1.8 in \cite{JacodShiryaev}, there exists a transition kernel $K_{t}'(\omega; d\xi)$ on $(\mathbb{R}^d, \ccB(\mathbb{R}^d))$ such that $\nu(\omega; dt, d\xi)=K_{t}'(\omega; d\xi)dA_{s,t}(\omega)$. Moreover, since by (\ref{c2})
\begin{eqnarray*}  
\sum_{i=1}^d C_{s,t}^{ii}(\omega) + A_{s,t}(\omega)= \sum_{i=1}^d\int_s^t a^{ii}(u, X_u(\omega))du 
\end{eqnarray*}
and since $(C_{s,t}^{ii}),  i\in \left\{1,..., d\right\}$ and $(A_{s,t})$ are non-negative increasing processes of finite variation, they are absolutely continuous with respect to Lebesgue measure. By Proposition I.3.13 in  \cite{JacodShiryaev} there exist predictable processes $\tilde{c}^{ii}$ and $\tilde a$ such that 
$$ C_{s,t}^{ii}=\int_s^t\tilde{c}_{u}^{ii}du \quad \text{and} \quad A_{s,t}=\int_s^t \tilde a_{u} du, \quad s \le t \le T. $$
Then $\widetilde{K}_{t}(\omega; d\xi):=\tilde a_{t}(\omega)K_{t}'(\omega;d\xi)$ is again a predictable transition kernel. It moreover satisfies  $\nu(\omega; dt, d\xi)=\widetilde{K}_{t}(\omega; d\xi)dt$ almost surely. 
This allows to obtain the following representation for \eqref{c2}, this time for all $1 \le i,j \le d$: 
\begin{equation*} 
C_{s,t}^{ij}(\omega)= \int_s^t \left( a^{ij}(u, X_u(\omega)) - \int_{\mathbb{R}^d}\xi_i\xi_j\widetilde{K}_u(\omega; d\xi)\right)du, \quad s \le t \le T.
\end{equation*}
Hence,  $(C_{s,t}^{ij})$ is also  absolutely continuous with respect to the Lebesgue measure for $ i \neq j$ yielding the representation 
$ C_{s,t}^{ij}= \int_s^t \tilde{c}_{u}^{ij}du.$
Following Proposition II.2.9 in \cite{JacodShiryaev}, we are able to find $c$ and $K$, such that  $\tilde{c}_{t}(\omega)=c(t,X_t(\omega))$ and $\tilde{K}_t(\omega; d\xi)=K_t(X_t(\omega); d\xi)$, showing (\ref{c3}) (and thus the validity of ii) and (\ref{c3_2}). \\
 
 iii): Inserting the results from ii) into \eqref{temp689} and equalling the predictable finite variation parts we obtain together with  continuity in $s$ that  
\begin{align}
\int_{\mathbb{R}^d} \sum_{|\mathbf{l}|=3}^{|\mathbf k|} \binom{\mathbf k}{\mathbf l} f_{\mathbf k- \mathbf l}(x)\xi^{\mathbf l} K_s(x;d\xi) \label{eq:xq}
  &= \ccH_sf_{\mathbf k}(x) - \sum_{i=1}^d D_if_{\mathbf k}(x)b^{i}(s,x) - \frac{1}{2}\sum_{i,j=1}^d D_{ij}f_{\mathbf k}(x) a^{ij}(s,x) \nonumber
\end{align}
for all $x \in S$.
Now for $| \mathbf k|=3$ this equation just reads
$$ \int_{\mathbb{R}^d} \xi^{\mathbf k} K_s(x;d\xi)\\
  = \ccH_sf_{\mathbf k}(x) - \sum_{i=1}^d D_if_{\mathbf k}(x)b^{i}(s,x) - \frac{1}{2}\sum_{i,j=1}^d D_{ij}f_{\mathbf k}(x) a^{ij}(s,x)$$
and here the right-hand side is in $\ccP_3(S)$ (as a function of $x$) and is continuous in $s$ (as follows immediately from the results proven above), showing \eqref{c4} for $|\mathbf k|=3$. The validity  for general $|\mathbf k|\geq 3$ now follows by induction.
\end{proof}

Our next aim is to reverse the above procedure. We show that a special semi-martingale satisfying i)-iii) of Theorem \ref{PolProcChar} is indeed $m$-polynomial, if the transition kernels $K_t$ satisfy an additional condition.

\begin{theorem}\label{thm:main1}
Let $X=(X_t)_{0 \leq t \leq T}$ be a Markov process with state space $S$ and let $m\geq 2$. Suppose that for all $ 0 \leq s \leq T$ and all $x \in S$ the process $(X_t)_{s \leq t \leq T}$ is a special semi-martingale with respect to $(\Omega,\ccF,\bbF_s,\mathbb{P}_{s,x})$ and that its characteristics $(B_s,C_s,\nu)$ satisfy  $i)-iii)$ of Theorem \ref{PolProcChar}. If
in addition
\begin{equation}\label{Mom1}
	\mathbb{E}_{s,x}\left[\int_{\mathbb{R}^d} |\xi|^m K_t(X_t,d\xi)\right]<\infty, \ \ \text{for almost all}\ \  0\leq s\leq t \leq T,
\end{equation}
or 
\begin{equation}\label{Mom2}
	\int_{\mathbb{R}^d} |\xi|^m K_t(X_t,d\xi)\leq \tilde{C}(1+|X_t|^m), \ \ 0 \leq t \leq T,
\end{equation}
for some constant $\tilde{C}$, then $X$ is an $m$-polynomial process  and for all $g \in \ccP_m(S)$, $0 \leq s <T$ and $x\in S$ it holds that
\begin{eqnarray}
\nonumber \ccH_sg(x)&=&\sum_{i=1}^d D_ig(x)b^i(s,x)+\frac{1}{2}\sum_{i,j=1}^d D_{i,j}g(x)c^{ij}(s,x) \\
\label{gen} &+& \int_{\mathbb{R}^d}(g(x+\xi)-g(x)-\sum_{i=1}^d D_i g(x)\xi_i)K_s(x, d\xi),
\end{eqnarray}
where $b^i, c^{ij}, K_s$ are defined in $i)-iii)$ of Theorem \ref{PolProcChar}.
\end{theorem}

\begin{proof}
Let $g \in \ccP_m$ and define 
$$V(x,\xi):= g(x+\xi)- g(x)- \sum_{i=1}^d D_i g(x)\xi_i, \quad x \in S, \ \xi \in \R^d.$$ 
By Taylor's formula $V(x,\xi)= \sum_{|\mathbf l|=2}^m D^{\mathbf l}g(x) \frac{\xi^{\mathbf l}}{\mathbf l!}$ and so 
$$|V(x,\xi)|  \leq \left(|\xi|^2\wedge |\xi|^m\right) \sum_{|\mathbf l|=2}^m \frac{|D^{\mathbf l}g(x)|}{\mathbf l !} =: \left(|\xi|^2\wedge |\xi|^m\right) \tilde{h}(x).$$  
In particular, (\ref{c2}) and (\ref{c4}) imply that
\begin{equation*}
\int_{\mathbb{R}^d}|V(x,\xi)|K_u(x, d\xi)\leq \tilde{h}(x)\int_{\mathbb{R}^d}\left(|\xi|^2\wedge |\xi|^m\right)K_u(x,d\xi) < \infty. 
\end{equation*}
Hence the process $(\int_s^{t}\int_{\mathbb{R}^d}V(X_u,\xi)K_u(X_u, d\xi)du)_{s \leq t \leq T}$ is of locally integrable variation.  It\^o's formula yields that
\begin{eqnarray}\label{LMM}
M_{s,t}^{g}&:=& g(X_t) - g(x) - \int_s^t \sum_{i=1}^d D_i g(X_u)b^{i}(u,X_u)du\\ 
\nonumber &-& \frac{1}{2}\int_s^t\sum_{i,j=1}^d D_{ij}g(X_u)c^{ij}(u,X_u)du - \int_s^t\int_{\mathbb{R}^d}V(X_u,\xi)K_u(X_u,d\xi)du, \quad s \le t \le T,
\end{eqnarray}
is a local martingale with respect to the stochastic basis $(\Omega,\ccF,\bbF_s,\mathbb{P}_{s,x})$ (compare the proof of Theorem II.2.42, $a) \Rightarrow c)$, in \cite{JacodShiryaev}). Define the generator
$$ \ccL_u g (x):= \sum_{i=1}^d D_i g(x)b^{i}(u,x) 
+ \frac{1}{2} \sum_{i,j=1}^d D_{ij}g(x)c^{ij}(u,x)+ \int_{\mathbb{R}^d}V(x,\xi)K_u(x,d\xi).$$
By validity of  $i)-iii)$ from Theorem \ref{PolProcChar},
$$ c^{ij}(u,x)=a^{ij}(u,x) - \int_{\R^d} \xi_i\xi_j K_u(x,d\xi).$$
Moreover, using again Taylors' formula,
$$ V(x,\xi) - \frac 1 2 \sum_{i,j=1}^d D_{ij}g(x)\xi_i\xi_j = \sum_{|\mathbf l|=3}^m D^{\mathbf l}g(x) \frac{\xi^{\mathbf l}}{\mathbf l!}$$
(with the convention that  $\sum_{|l|=3}^m (\ldots):=0$ if $m=2$). Hence,
\begin{equation*}
  \ccL_u g (x) = \sum_{i=1}^d D_i g(x)b^{i}(u,x) 
+ \frac{1}{2} \sum_{i,j=1}^d D_{ij}g(x)a^{ij}(u,x) 
+\int_{\mathbb{R}^d}\bigg( \sum_{|\mathbf l|=3}^m D^{\mathbf l}g(x) \frac{\xi^{\mathbf l}}{\mathbf l!}\bigg) K_u(x,d\xi). 
\end{equation*}
  Conditions  $i) -iii)$ from Theorem \ref{PolProcChar} imply that $\ccL_u g \in \ccP_m(S)$ and that $u \mapsto \ccL_ug(x)$ is continuous. Hence there exists a continuous non-negative function $\beta(u)$ and a constant $\tilde C$ such that $|\ccL_ug(x)| \leq \tilde C \beta(u) (1+|x|^m)$.

Next, we show that $M^g$ is actually a true martingale. 
 Using that  $|g(y)| \leq C'(1+|y|^m)$ for some constant $C'$ we obtain
\begin{eqnarray}\lefteqn{
 \mathbb E_{s,x} \left[\sup_{s \leq \tau \leq t} |M_{s,\tau}^g|\right]
\leq \mathbb E_{s,x}[|g(X_\tau)|] + |g(x)| + \tilde C \int_s^\tau \beta(u) \: \mathbb E_{s,x}[|X_u|^{m}]du} \qquad \notag\\
&\leq& C'(1+ \mathbb E_{s,x}[|X_\tau|^m|]) + |g(x)| + \tilde C \int_s^t \beta(u) du \cdot  \mathbb E_{s,x}\left[\sup_{s\leq u \leq \tau}|X_u|^{m}\right].\label{eq:812}
\end{eqnarray}
Now,  
\begin{equation*}
  E_{s,x}\left[\sup_{s\leq u \leq \tau}|X_u|^{m}\right] < \infty
\end{equation*}
 follows from assumptions (\ref{Mom1}) and/or (\ref{Mom2}) as in \cite{CuchieroKellerResselTeichmann2012} by an application of the Burkholder-Davis-Gundy inequality and we provide a precise proof in Lemma \ref{lem:moment} below. 
 Since $M^g$ is a local martingale, the resulting finiteness of \eqref{eq:812} implies that is indeed a true martingale.

As in the proof of $i) \Rightarrow ii)$ in Lemma \ref{lem:martingale} we see that 
 $$ P_{s,t}g(x)= g(x) + \int_s^t P_{s,u} \ccL_u g(x) du.$$ 
 In particular, as in the proof of Theorem \ref{thm:mainequivalence}, $u \mapsto P_{s,u}\ccL_ug(x)$ turns out to be continuous and hence the following limit exists:
\begin{eqnarray*}
 \lim_{ h \searrow 0} h^{-1} \left( P_{s,s+h}g(x)-g(x) \right) = \ccL_s g(x).
\end{eqnarray*}
Hence, $g \in \ccD(\ccH_s)$ and  $\ccH_sg(x) = \ccL_sg(x)$. 
Summarizing, we have shown that all conditions $a) -d)$ in Theorem \ref{thm:mainequivalence} are satisfied and therefore $X$ is $m$-polynomial.  
\end{proof}

The proof of Theorem \ref{thm:main1} relied on the following lemma, which is an adaption of
Lemma 2.17 in \cite{CuchieroKellerResselTeichmann2012} to the time-inhomogeneous setting. 
\begin{lemma}\label{lem:moment}
Let $0 \leq s < \tau \leq T$ and $x \in E$ be fixed. Consider a semi-martingale $X$ and assume that its semi-martingale characteristics $(B,C,\nu)$ satisfy points $i)-iii)$ of Theorem \ref{PolProcChar}. Then there exists a constant $\tilde{C}_0$ such that
\begin{align}\label{eq:moment1}
	\mathbb{E}_{s,x}\left[\sup_{s \leq u\leq \tau} |X_u|^m\right]& \leq \tilde{C}_0\left( |x|^m + 1 + \int_s^\tau\mathbb{E}_{s,x}\left[\int_{\mathbb{R}^d} |\xi|^m K_u(X_u,d\xi)\right] du \right.\nonumber\\
	 &\qquad \left. {} + \int_s^\tau \mathbb{E}_{s,x}\left[|X_u|^m\right]du \right).
\end{align}
Moreover, if one of the conditions (\ref{Mom1}) or (\ref{Mom2}) is satisfied, then there exist finite constants $\tilde C_1$ and $\tilde C_2$ such that
\begin{equation}\label{eq:moment2}
	\mathbb{E}_{s,x}\left[\sup_{s \leq u\leq \tau} |X_u|^m\right]\leq  \tilde C_1 e^{\tilde C_2\tau}.
\end{equation}
\end{lemma} 
\begin{proof}
Define for $s\leq t <T$ the  process $Y$ by $Y_t=(t,X_t)$ 
with state space $\tilde{E}:= [s,T)\times \mathbb{R}^d$ and sample space $\tilde{\Omega}:= [s,T)\times \Omega$. The process $Y$ is a space-time homogeneous Markov process. In particular, this process is a $(d+1)$-dimensional semi-martingale whose characteristics $(\tilde{B},\tilde{C},\tilde{\nu})$ are defined by:
\begin{eqnarray*}
\tilde{B}_t&=& \int_s^t\tilde{b}(Y_u)du = \int_s^t b(u,X_u)du \\
\tilde{C}_t&=& \int_s^t\tilde{c}(Y_u)du = \int_s^t c(u,X_u)du \\
\tilde{\nu}(\tilde{\omega};dt,d\xi)&=& \tilde{K}(Y_t(\tilde{\omega}),d\xi)dt= K_t(X_t(\omega),d\xi)dt.
\end{eqnarray*}
Since $b$, $c$ and $K$ satisfy $i)-iii)$ in Theorem \ref{PolProcChar}, the characteristics $(\tilde{B}, \tilde{C}, \tilde{\nu})$ satisfy the conditions (2.11)-(2.13)  from Proposition 2.12 in \cite{CuchieroKellerResselTeichmann2012}. Then,  Lemma 2.17 in \cite{CuchieroKellerResselTeichmann2012} yields (\ref{eq:moment1}) and (\ref{eq:moment2}).
\end{proof}

\subsection{Polynomial jump-diffusions}

We conclude this section by introducing a class of processes to which Theorem \ref{thm:main1} applies.

\begin{definition}\label{PolJumpDiff} 
  Let $m \geq 2$ and let $X=(X_t)_{0 \leq t \leq T}$ be an $S$-valued semi-martingale with differential characteristics $(b,c,K)$. Then we call $X$ an \emph{$m$-polynomial jump-diffusion} if the following holds:
  \begin{enumerate}[i)]
      \item For all $1 \leq i \leq d$, $0 \leq t < T$ and all $x \in S$: 
$$ b^i(t,.) \in \ccP_1(S) \quad \text{and} \quad b^i(.,x) \in C[0,T).$$
  \item For all  $1 \leq i,j \leq d$, $0 \leq t < T$ and all $x \in S$: 
$$ c^{ij}(t,.) \in \ccP_2(S) \quad \text{and} \quad c^{ij}(.,x) \in C[0,T).$$
 \item For all $2 \leq |\mathbf k | \leq m$ there exist $\alpha_{\mathbf l} \in C[0,T), 0 \leq |\mathbf l| \leq |\mathbf k|$ such that
$$ \int_{\R^d} \xi^{\mathbf k}  K_t(x,d\xi) = \sum_{|\mathbf l|=0}^{|\mathbf k|} \alpha_{\mathbf l}(t) x^{\mathbf l}, \qquad x \in S, \quad t \in [0,T).$$
\item Either \begin{equation*}
	\mathbb{E}_{s,x}\left[\int_{\mathbb{R}^d} |\xi|^m K_t(X_t,d\xi)\right]<\infty, \ \ \text{for almost all}\ \  0\leq s\leq t \leq T,
\end{equation*}
or
\begin{equation*}
	\int_{\mathbb{R}^d} |\xi|^m K_t(X_t,d\xi)\leq \tilde{C}(1+|X_t|^m), \ \ 0 \leq t \leq T.
\end{equation*}
  \end{enumerate}
\end{definition}
Noting that the above assumptions ii) and iii), for $|\mathbf k|=2$, imply that the functions
$$ a^{ij}(u,y)=c^{ij}(u,y) + \int_{\R^d } \xi_i\xi_j K_u(y,d\xi)$$
satisfy all requirements of Theorem \ref{PolProcChar}, part i), the following corollary immediately follows from Theorem \ref{thm:main1}.
\begin{corollary}\label{cor:jump} 
 Let $X=(X_t)_{0 \leq t \leq T}$ be an $S$-valued $m$-polynomial jump-diffusion. Then $X$ is an $m$-polynomial process and its infinitesimal generator $(\ccH_s)_{0 \le s \le T}$ satisfies 
\begin{eqnarray*}
\nonumber \ccH_sf(x)&=&\sum_{i=1}^d D_if(x)b^i(s,x)+\frac{1}{2}\sum_{i,j=1}^d D_{i,j}f(x)c^{ij}(s,x) \\
&+& \int_{\mathbb{R}^d}(f(x+\xi)-f(x)-\sum_{i=1}^d D_i f(x)\xi_i)K_s(x, d\xi),
\end{eqnarray*}
for all $f \in \ccP_m(S)$, $0 \leq s <T$ and $x\in S$.
\end{corollary}

\section{Examples}\label{sec:examples}

In this section we use the characterization results discussed above to present some examples of polynomial processes, as well as examples of stochastic processes which do not fall into this category.\\
 
We start by showing that some widely used and well-known classes of stochastic processes are indeed  polynomial. This indicates the wide applicability of polynomial processes.

\subsection{Polynomial diffusions}\label{subsec:polydif}    

An $S$-valued semi-martingale $(X_t)_{0 \leq t \leq T}$  is called a \emph{polynomial diffusion} if it is a polynomial It\^o process without jumps, i.e.\ $K_t=0$ for all $0 \leq t \leq T$. 
By Corollary \ref{cor:jump}, a polynomial diffusion is a polynomial process. In particular, if $c=\sigma\sigma^\top,$ with $\sigma=\sigma(t,x) \in \R^{d \times d}$ such that
$$ \sigma^{ij}(t,.) \in \ccP_1(S) \quad \text{and} \quad \sigma^{ij}(.,x) \in C[0,T]$$
for all $1 \leq i,j \leq d, x \in S$ and $0 \leq t < T$, then $X$ is a strong solution of the stochastic differential equation
$$ dX_t = b(t,X_t)dt + \sigma(t,X_t) dW_t,$$
where $W$ is standard $d$-dimensional Brownian motion. Recall that by definition, $b^i(t,.)\in \ccP_1(S)$ for all $0 \le t \le T$.

\subsection{Time inhomogeneous L\'evy processes}
An $\R^d$-valued process $(X_t)_{0 \leq t \leq T}$ is called time-homogeneous L\'evy process if it has independent increments and for every $0 \leq t \leq T$ the law of $X_t$ is characterized by the characteristic function
\begin{align*}
 \mathbb E \left[ e^{i \langle u , X_t \rangle}\right] 
&= \exp \left( \int_0^1 \left[ i \langle u , b_s \rangle - \frac 1 2 \langle u , c_s u \rangle  + \int_{\R^d} \left( e^{i\langle u, \xi\rangle} - 1 -i \langle u,\xi \rangle  1_{\{|\xi| \leq 1\}}\right) K_s(d\xi) \right] ds\right),
\end{align*}  
where for every $0 \leq s \leq T$ it holds that $b_s \in \R^d, c_s \in \R^{d \times d}$ is symmetric and positive definite and $K_s$ is a measure on $\R^d$, see \cite{Kluge05}. Then $X$ is a semi-martingale with differential characteristic $(b,c,K)$, see \cite{JacodShiryaev}. In particular, if $s \mapsto b_s$ and $s \mapsto c_s$ are continuous and if $F$ satisfies condition (\ref{Mom1}) or (\ref{Mom2}) for some $m \geq 2$, then $X$ is $m$-polynomial.

\subsection{Affine processes} 

A stochastically continuous affine process $X$ on $S=\mathbb{R}_+^p \times \mathbb{R}^{d-p}$ is m-polynomial if the killing rate is constant and if the L\'evy measures $\nu_i(t,\cdot)$, for $i \in \{1,...,p\}$, satisfy
\begin{equation}
  \label{eq:43}
\int_{|\xi|>1}|\xi|^m \nu_i (t, d\xi)<\infty, \qquad t \ge 0.
\end{equation}

Indeed, by Theorem 2.13 in \cite{Filipovic05}  we know that a strongly regular affine process is a Feller semi-martingale with infinitesimal generator given by: 
\begin{eqnarray}
\nonumber \ccH f(x)&=& \sum_{k,l=1}^d A_{kl}(t,x) D_{k,l}^2 f(x) + \left\langle B(t,x),\nabla f(x)\right\rangle - C(t,x)f(x) \\
\nonumber&+& \int_{E\setminus \left\{0\right\}} \left(f(x+\xi)-f(x)-\left\langle \nabla f(x),\chi(\xi)\right\rangle\right)  M(t,x,d\xi),
\end{eqnarray} 
where
\begin{eqnarray} 
\nonumber	A(t,x)&=& a(t)+\sum_{i=1}^{m} \alpha_i(t) x_i, \ \ \ a(t),\alpha_i(t)\in\mathbb{R}^{d\times d},\\
\nonumber	B(t,x)&=& b(t) + \sum_{i=1}^n \beta_i(t) x_i, \ \ \ b(t), \beta_i(t) \in \mathbb{R}^d, \\ 
\nonumber	C(t,x)&=& c(t) + \sum_{i=1}^m \gamma_i(t) x_i, \ \ \ c(t),\gamma_i(t) \in\mathbb{R}_+, \\
\nonumber	M(t,x,d\xi)&=& \lambda(t,d\xi) + \sum_{i=1}^p \nu_i(t,d\xi)x_i,
\end{eqnarray}
and where $\lambda(t,\cdot)$ is a Borel measure on $E\setminus \left\{0\right\}$. In particular, the differential characteristic are affine functions in $X$. The properties of the coefficients are characterized by the so-called admissibility condition and continuity. For details we refer to \cite{Filipovic05}. Note that here we assume in addition continuous differentiability. 
(Time-inhomogeneous) affine processes without the assumption of stochastic continuity have bee studied in \cite{KellerResselSchmidtWardenga2018}. 

\pagebreak

\subsection{Examples of time-inhomogeneous processes which are not m-polynomial}

\subsubsection*{A process with quadratic drift} Let $S\subset \mathbb{R}$ be closed. 
Let $X$ be the  solution of the SDE (until explosion)
$$dX_t=(a(t)+b(t)X_t + X_t^2)dt + dW_t, \ \ X_0=x,$$
where $a(t)$ and $b(t)$ are continuous and bounded functions. The generator of the process can be easily obtained and is given by the next expression,
$$\ccH_tf(x)=(a(t)+b(t)x+x^2)\frac{df(x)}{dx}+\frac{1}{2}\frac{d^2f(x)}{dx^2}.$$
But this operator maps polynomials of degree $m$ into polynomials of degree $m+1$. Namely, if we take $f_m(x)=x^m\in\ccP_m$ and apply the operator $\ccH_t$, then
$$\ccH_tf_m(x)=ma(t)x^{m-1}+mb(t)x^m+mx^{m+1}+\frac{m(m-1)}{2}x^{m-2}.$$
We can see the that $\ccH_tf_m\in\ccP_{m+1}$. So we conclude that $\ccH_t\left(\ccP_m\right)\subset\ccP_{m+1}$ and therefore the process $X$ is not $m-$polynomial. 

\subsubsection*{The Cauchy process}
Following Example 1.12.2 from \cite{GulisashviliCasteren2006}, we call an $\R^d$-valued Markov process a  Cauchy process if its transition density is given by 
$$p_{s,t}(x,y)=\Gamma\left(\frac{d+1}{2}\right)\frac{t-s}{\left[\pi\left((t-s)^2+\left|y-x\right|^2\right)\right]^{\frac{1}{2}(d+1)}}, \ \ 0\leq s<t, x\in\mathbb{R}^d, y\in\mathbb{R}^d.$$
For the sake of simplicity in the next computation we are just considering the case $d=1$, $x=0$ and $s=0$. First, we are going to take $f_1(x)=x \in \ccP_1(\mathbb{R})$ and compute for $t>0$ 
\begin{eqnarray}
\nonumber P_{0,t}f_1(0)&=&\mathbb{E}[X_t|X_0=0]\\
\nonumber &=& \frac{t}{\pi}\int_{\mathbb{R}_+}\frac{y}{t^2+y^2}dy
= \frac t \pi \left.\frac{\ln (t^2+y^2)}{2}\right|_0^{\infty}
= \infty.
\end{eqnarray}
The same result would be obtained for every $f(x)\in \ccP_k$, $k=0,...,m$ and for every $m\in\mathbb{N}$, since it is known that the Cauchy distribution has infinite moments of every order. In particular, the Cauchy process provides a further example of a Markov process which is not $m$-polynomial for any $m \in \N$.

\subsection{An $m$-polynomial process which is not affine}
We consider again the one-dimensional stochastic process given by Equation (\ref{easyexample}),
	\[ X_t=\int_s^ta(u)du +W_t
\]
and we also consider, for some constant $A_0,A_1,A_2$, the process
	\[ Y_t=A_0 + A_1X_t + A_2 X_t^2.
\]
We will now show that the process $(X,Y)$ is a two-dimensional polynomial process, which is not affine.
From It\^o's lemma, we know that the dynamics of $Y_t$ is given by
	\[ dY_t=(A_1+ A_2+ 2A_2X_t)a(t)dt+ (A_1 + 2A_2X_t)dW_t
\]
and hence 
	\[\begin{pmatrix} dX_t\\dY_t \end{pmatrix} = \left( \begin{pmatrix} a(t)\\ (A_1 + A_2)a(t) \end{pmatrix}  +\begin{pmatrix} 0\\ 2A_2a(t)X_t \end{pmatrix} \right)dt + \begin{pmatrix} 1\\ A_1 + 2A_2X_t\end{pmatrix}dW_t.\]
As we have discussed in Section \ref{subsec:polydif}, this is the dynamics of a polynomial diffusion. Let us compute its second characteristic: Since
$\begin{pmatrix} 1\\ A_1 + 2A_2X_t\end{pmatrix}$ is the standard deviation, the variance is given by

	\[\begin{pmatrix} 1\\ A_1 + 2A_2X_t\end{pmatrix} \begin{pmatrix} 1 &  A_1 + 2A_2X_t\end{pmatrix}=\begin{pmatrix} 1& A_1 + 2A_2X_t \\ A_1 + 2A_2X_t & (A_1 + 2A_2X_t)^2 \end{pmatrix}.
\]
It can be expressed as:
	\[ \begin{pmatrix} 1& A_1\\ A_1  & A_1^2 \end{pmatrix} + \begin{pmatrix} 0& 2A_2\\ 2A_2  & 4A_1A_2 \end{pmatrix}X_t +  \begin{pmatrix} 0& 0\\ 0  & 4A_2^2 \end{pmatrix}X_t^2.
\]
Then the second characteristic is:
	\[C_t=\int_0^t\left[ \begin{pmatrix} 1& A_1\\ A_1  & A_1^2 \end{pmatrix} + \begin{pmatrix} 0& 2A_2\\ 2A_2  & 4A_1A_2 \end{pmatrix}X_u +  \begin{pmatrix} 0& 0\\ 0  & 4A_2^2 \end{pmatrix}X_u^2 \right]du,
\]
Since the second characteristic is a quadratic function, it is not an affine process.

\subsection{A Jacobi-process with jumps}

Energy markets have some interesting stylized facts: they inhibt strong seasonal effects, have upward and downward jumps / spikes. Most notably the spot price of electricity has upper and lower bounds. Time-inhomogeneous polynomial processe are ideally suited to capture all these effects as we now show. A more detailed application can be found in \cite{Agoitia2017}.

In fact, given our results above we can use a variant of the Jacobi process with jumps (here de-seasonalized)
$$ dS_t = \kappa (\theta - S_t) dt  + \sqrt{ S_t (1-S_t)} dW_t + dJ_t $$
    where (downward jumps, $-1 \le a < b < 0$)
    $$ K(t,d\xi) =  \Ind_{[ ax,bx]}(\xi) \frac{-1}{\log \nicefrac a b} \xi^{-1} d \xi $$
    or ($\alpha \in (0,1)$)
    $$ K'(t,d\xi) = \Ind_{[ 1-x, \alpha (1-x)]}(\xi) \frac{1}{\log \nicefrac 1 \alpha}\xi^{-1} d \xi $$
    (and linear combination of these). 
    Many further examples and a detailed study of these kind of processes may be found in \cite{CuchieroLarssonSvaluto2017}.

\section{Polynomial processes: Computation}\label{sec:computation}

One of the important properties of $m$-polynomial processes is the fact that their moments
$$ \mathbb E_{s,x}[X_t^{\mathbf k}]=P_{s,t}f_{\mathbf k}(x), \qquad f_{\mathbf k}(x)=x^{\mathbf k}, \quad \mathbf k \in \mathbb N_0^d,$$
can be computed in a fairly simple manner, because the Kolmogorov  equations  reduce this problem to the computation of a solution of an ordinary linear differential equation. In this section we will show this for time-inhomogeneous polynomial processes.

\subsection{Representing matrices}\label{sec:matrices}
Following \cite{Kreyszig89} Section 2.9 we recall the well-known concept of representing matrices for a linear operator. In this regard,  consider a finite-dimensional $\R$-vector space $V$ with basis $(v_1,\ldots,v_N)$ and let $L$ be a linear operator on $V$. Since every $u \in V$ has a unique representation  $u= \sum_{j=1}^N u_jv_j$ with $u_1,\dots,u_N \in \R,$ by linearity
$ Lv = \sum_{j=1}^N u_j Lv_j.$ From this we infer that $L$ is uniquely determined by $Lv_j, j=1,\ldots,N$. Since  $Lv_j \in V$,  there exist unique coefficients $l_{ij}$ such that
\begin{equation}
  \label{eq:13}
 Lv_j = \sum_{i=1}^N l_{ij} v_i  
\end{equation} 
and the matrix $\mathbf L = (l_{ij})_{i,j=1}^N \in \R^{N \times N}$ is called the \emph{representing matrix} of $L$ with respect to the given basis of $V$.

For any matrix $\bA\in \R^{N \times N}$ we define the \emph{spectral norm}   
$ \|\bA\|_2:= \max_{|x|=1} |\mathbf Ax|. $ Note that the spectral norm of $\bL$ does depend on the chosen basis. This gives us the freedom to make the spectral norm of $\bL$ arbitrary small: indeed, consider $\varepsilon>0$ and the basis $\varepsilon^{-1}(v_1,\dots,v_N)$. By \eqref{eq:13}, we obtain that the representation of $\bL$ with respect to this basis, denoted by   $\bL^\varepsilon $ satisfies
$ \| \bL^\varepsilon \|_2 = \varepsilon \| \bL \|_2.$ 

\begin{remark} \label{rem:spectral}
Note that
\begin{align}
  \label{eq:20}
  \|\bA\|_2 &= \sqrt{\lambda_{\max}(\bA^\top \bA)} , 
\end{align}
i.e. the norm is given by the square root of the maximal eigenvalue of the positive definite symmetric matrix $\bA^\top\bA$. In view of (\ref{eq:20}) the norm $\|\bA\|_2$ is called the {\bf spectral norm} of $\bA$.
Moreover, 
\begin{equation}
  \label{eq:19}
  \max_{ij}|a_{ij}| \leq \|\bA\|_2 \leq n  \max_{ij}|a_{ij}|.
\end{equation}
In particular, a map $\R \ni t \mapsto \bA_t=(a_{ij}(t)) \in (\R^{N \times N},\|.\|_2)$ is continuous or differentiable if and only if this is true of all coefficient functions $a_{ij}(t)$,  see  \cite{Teschl2012}.  
\end{remark}

\subsection{Representations of polynomial processes}
We continue in the setting of the previous sections and study a polynomial process $X$ on the closed state space $S \subset \mathbb R^d$ and denote its transition operators by $(P_{s,t})_{(s,t) \in \Delta}$ and its infinitesimal generator by $(\ccH_s)_{0 \leq s < T}$. 

Fix a $0 \leq k \leq m$ and consider a fixed basis of $\ccP_k(S)$. For each $(s,t)\in \Delta$, we denote by   $\mathbf P _{s,t}$, and $\mathbf H_s$ the representing matrices of $P_{s,t}$ and $\ccH_s$.

\begin{proposition}
 With the assumptions and notation from above the following holds:
 \begin{enumerate}[i)]
 \item The map $[0,T) \ni s \mapsto \mathbf H_s$ is continuous.
 \item For every $t \in [0,T)$ the map
$ [0,t) \ni s \mapsto \mathbf P_{s,t}$ is right-differentiable and 
   \begin{equation}
\frac{d^+}{ds} \mathbf P_{s,t}= -\mathbf H_s \mathbf P_{s,t}.  \label{backward:mat}   
   \end{equation}
\item For every $s \in [0,T)$ the map
$ [s,T) \ni t \mapsto \mathbf P_{s,t}$ is right-differentiable and 
 \begin{equation}
\frac{d^+}{dt} \mathbf P_{s,t}= \mathbf P_{s,t} \mathbf H_{t}.     \label{forward:mat}   
   \end{equation}
 \end{enumerate}
\end{proposition}
\begin{proof}
Suppose that $\mathbf P_{s,t}=(p_{s,t}^{ij})_{i,j=1}^N$ and $\mathbf H_s= (a_s^{ij})_{i,j=1}^N$ are representing matrices of $P_{s,t}$ and $\mathcal H_s$ with respect to a given basis $(v_1,\ldots,v_N)$ of $\ccP_k(S)$. Note that
\begin{equation*} 
p_{s,t}^{ij}= \langle e_i | \bP_{s,t} e_j  \rangle \quad \text{and} \quad a_{s}^{ij} = \langle e_i | \bH_s e_j \rangle,  
\end{equation*}
where $(e_j)_{j=1}^N$ denotes the standard basis of $\mathbb R^N$. Let the invertible linear map $U : \ccP_k(S) \to \mathbb R^N$ be defined by $Uv_j=e_j$. Then 
$$\mathbf P_{s,t}= U P_{s,t} U^{-1} \quad \text{ and } \quad \mathbf H_s = U \ccH_s U^{-1}.$$  
By Proposition \ref{Thm:gen}, the map $s \mapsto \ccH_s v_j = \ccH_s U^{-1}e_j \in (\ccP_k(S),\|.\|_k)$ is continuous for every $j=1,\ldots,N$. Since the linear operator $U$ is bounded, the maps
$$ s \mapsto a_{s}^{ij} = \langle e_i| U \ccH_s U^{-1} e_j \rangle, \qquad 1 \leq i,j \leq N$$
are continuous as well. Part i) thus follows from Remark \ref{rem:spectral}. For part ii) and iii) we can argue analogously. For instance, for $h>0$ small enough we can compute
\begin{eqnarray*}
\frac{d^+}{ds} p_{s,t}^{ij} &=& \lim_{ h \searrow 0}\frac 1 h (p_{s+h,t}^{ij}-p_{s.t}^{ij}) \\
&=& \lim_{ h \searrow 0} h^{-1} \langle e_i | \left( \mathbf P_{s+h,t}- \mathbf P_{s,t} \right)e_j \rangle \\
&=& \lim_{ h \searrow 0} h^{-1} \langle e_i| U  \left( P_{s+h,t}- P_{s,t} \right)U^{-1}e_i \rangle
\end{eqnarray*}  
and by (\ref{Kolmogorovbackward}) the right-hand side tends to 
$$ -\langle e_i| U \ccH_s P_{s,t} U e_j \rangle,$$
which is the $ij$-th component of $-\bH_s \bP_{s,t}$. This shows the validity of ii) and the validity of iii) follows in exactly the same way.
\end{proof}
The previous proposition shows that we can compute the action of $P_{s,t}$ on polynomials by solving the ordinary differential equations (\ref{backward:mat}) or (\ref{forward:mat}). In the time-homogeneous case,  where the pair $(P_{s,t},\ccH_s)$ is replaced with the pair $(P_s,\ccH)$, this is straightforward: 
\begin{eqnarray*}
&& \frac{d^+}{ds} \mathbf P_{s}= -\mathbf H \mathbf P_{s} \quad\text{and} \quad \bP_0=\mathbf I\quad \\
&\Rightarrow& \quad \bP_{s}= e^{-s\bH}:=\sum_{k=0}^\infty \frac{\bH^k}{k!},
\end{eqnarray*}
see \cite{Teschl2012} Chapter 3.

In the inhomogeneous case, however, the situation is more complicated: while  the linear equations (\ref{backward:mat}) and (\ref{forward:mat})  are solvable because $\bH_s$ is continuous (see \cite{Teschl2012} Theorem 3.9), there is no simple explicit form of the solution.  In the following, we will present a method of Wilhelm Magnus, see \cite{Magnus}, which will at least allow us to obtain approximate solutions.\\

Denote by
$ [\mathbf A, \mathbf B]:=\mathbf A \mathbf B - \mathbf B \mathbf A$ 
the \emph{commutator} of two matrices.
 \begin{proposition}\label{thm:magnus1}
Let $[0,T) \ni t \mapsto \mathbf B(t) \in \R^{N \times N}$ be continuous and suppose that
$  \int_0^T \|\mathbf B(t)\|_2 dt < \pi.$
Consider the initial value problem 
\begin{equation}
  \label{eq:22}
  \begin{cases}
\frac d {dt} \mathbf U(t) \,= \mathbf B(t)\mathbf U(t), \quad &t> 0 \\
 \phantom{\frac d {dt} }\mathbf U(0)=\mathbf I.       &
  \end{cases}
\end{equation}
Then the unique solution $\mathbf U(t) \in \R^{N \times N}$ of  this problem  is of the form $\mathbf U(t)=e^{\mathbf \Omega(t)}$ where $\mathbf \Omega(t) \in \R^{N \times N}$ is expressible as an absolutely convergent power series
\begin{equation}
  \label{eq:17}
\mathbf \Omega(t) = \sum_{k=1}^\infty \mathbf \Omega_k(t)
\end{equation}
whose coefficients depend on $t$ and $(\mathbf B(s))_{0 \leq s \leq t}$. The first three terms of this series are given by
\begin{eqnarray*}
\mathbf \Omega_1(t)&=&\int_0^t\mathbf B(u) du\\
\mathbf \Omega_2(t)&=&\frac{1}{2}\int_0^t du\int_0^u dv [\mathbf B(u),\mathbf B(v)] \\
\mathbf \Omega_3(t)&=&\frac{1}{6}\int_0^t du \int_0^udv \int_0^v dw \left([\mathbf B(u),[\mathbf B(v),\mathbf B(w)]]+[\mathbf B(w),[\mathbf B(v),\mathbf B(u)]]\right).
\end{eqnarray*}
 \end{proposition}
For a proof of this result, see  \cite{Blanes} Theorem 9 and \cite{Sanchez}.
The series (\ref{eq:17}) is called {\bf Magnus series} or {\bf Magnus expansion}. 
 \begin{remark}\label{EasyMagnus} 
One can give explicit but quite involved formulas for all of the $\mathbf \Omega_k$ (see \cite{Blanes}). Here we will typically neglect the terms of order $4$ and higher and use $$ e^{\left( \sum_{k=1}^3 \mathbf \Omega_k(t)\right)}$$
as an approximative solution to (\ref{eq:22}). However, let us note that the higher order terms are expressed in terms of more and more nested commutators of the family $(\mathbf B(t))$, which in particular shows that if this family commutes, i.e. $[\mathbf B(t), \mathbf B(s)]=0$ for all $s,t$, then $\mathbf\Omega^{\mathbf B}_k(t)=0$ for $k > 1$ and so the solution of (\ref{eq:22}) is given by
\begin{equation}
  \label{eq:23}
  \mathbf U(t)= e^{\int_0^t \mathbf B(r) dr}. 
\end{equation}  
see \cite{Blanes}.
\end{remark}
Now let us apply Theorem \ref{thm:magnus1} to the Kolmogorov equations (\ref{backward:mat}) and (\ref{forward:mat}). Recall that by choosing an appropriate basis we are always able to achieve \eqref{condition:pi}, see Section \ref{sec:matrices}.

\begin{theorem}\label{thm:magnus2} If 
\begin{align}\label{condition:pi} \int_s^T\|\mathbf H_s\|_2 ds < \pi,\end{align}
 then for $0 \leq s \leq t \leq T$ we have $\mathbf P_{s,t}=e^{\mathbf \Omega(s,t)}$ where $\mathbf \Omega(s,t)$ is expressible as an absolutely convergent power series
 \begin{equation}
    \label{eq:24}
\mathbf \Omega(s,t) = \sum_{k=1}^\infty \mathbf \Omega_k(s,t)
\end{equation}
whose coefficients depend on $s,t$ and $(\mathbf H_u)_{s \leq u \leq t}$. The first three terms of this series are given by
{\small
\begin{eqnarray}
\mathbf \Omega_1(s,t)&=&\int_s^t\mathbf H_u du \label{eq:omega1}\\
\mathbf \Omega_2(s,t)&=&-\frac{1}{2}\int_s^t du\int_s^u dv \:[\mathbf H_u,\mathbf H_v] \label{eq:omega2}\\ 
\mathbf \Omega_3(s,t)&=&\frac{1}{6}\int_s^t du \int_s^udv \int_s^v dw \left([\mathbf H_u,[\mathbf H_v,\mathbf H_w]]+[\mathbf H_w,[\mathbf H_v,\mathbf H_u]]\right). \label{eq:omega3}
\end{eqnarray}
}
\end{theorem}
As a special case we obtain for a commuting  family $(\mathbf H_s)$, that 
\begin{equation}
  \label{eq:28}
  \mathbf P_{s,t}= e^{\int_s^t \mathbf H_u du}.
\end{equation}
\begin{proof}
  The theorem follows by applying Theorem \ref{thm:magnus1} to the forward equation \eqref{forward:mat} and we aim at a representation of the from \eqref{eq:22}. 
To this end, fix $s \in [0,T]$ and 
$$ \mathbf U(r):= \mathbf P_{s,s+r}^\top \quad \text{ and } \quad \mathbf B_s(r):= \mathbf H_{s+r}^\top.$$
Then $\mathbf U(0)=\mathbf P_{s,s}^\top=\mathbf I^\top= \mathbf I$ and for $0 < r < T-s$
$$ \frac {d^+} {dr} \mathbf U(r)= \left( \frac {d^+} {dr} \mathbf P_{s,s+r} \right)^\top = \left( \mathbf P_{s,s+r} \bH_{s+r} \right)^\top =\bH_{s+r}^\top \mathbf P_{s,s+r}^\top = \mathbf B_s(r) \mathbf U(r).$$
Since by assumption
$$ \int_0^{T-s} \|\mathbf B_s(r)\|_2 dr = \int_0^{T-s} \| \mathbf H_{s+r}^\top\|_2 dr = \int_s^\top \| \mathbf H_u\|_2 du < \pi,$$
Theorem \ref{thm:magnus1} yields that the unique solution is of the form $\mathbf U(r) = e^{\tilde {\mathbf \Omega}(r)}$ 
with $\tilde  {\mathbf \Omega}$ having an expression of a power series as in \eqref{eq:17}.
This in turn implies that, by letting $t=s+r$, 
$$ \mathbf P_{s,t}= \big( e^{\tilde{\mathbf \Omega}(t-s)} \big)^\top=e^{(\tilde{\mathbf \Omega}(t-s))^\top}.$$
In particular, setting 
$ \mathbf \Omega(s,t):=[\tilde{\mathbf \Omega}(t-s)]^\top, $
the first part of the theorem follows. Moreover, we can compute
$$ \mathbf \Omega_1(s,t) = \left[ \int_0^{t-s} \mathbf B_s(u) du\right]^\top = \left[\int_0^{t-s} \mathbf H_{u+s}^\top du\right]^\top = \int_s^t \mathbf H_u du,$$
showing (\ref{eq:omega1}). To compute $\mathbf \Omega_2(s,t)$ we use that for two matrices $\mathbf C, \mathbf D$ we have $[\mathbf C, \mathbf D]^\top = - [ \mathbf C^\top, \mathbf D^\top]$ and obtain
\begin{eqnarray*}
\mathbf \Omega_2(s,t) &=& \left[\tilde{\mathbf \Omega}_2(t-s)\right]^\top =\frac{1}{2} \left[\int_0^{t-s} du\int_0^u dv [\mathbf B_s(u),\mathbf B_s(v)] \right]^\top\\  
&=& - \frac{1}{2}\int_0^{t-s} du\int_0^u dv [\mathbf B_s^\top(u),\mathbf B_s^\top(v)]   
= - \frac{1}{2}\int_0^{t-s} du\int_0^u dv [\mathbf H_{s+u},\mathbf H_{s+v}] \\  
&=& - \frac{1}{2}\int_s^{t} du\int_0^{u-s} dv [\mathbf H_{u},\mathbf H_{s+v}]  = - \frac{1}{2}\int_s^{t} du\int_s^{u} dv [\mathbf H_{u},\mathbf H_{v}],
\end{eqnarray*}
showing (\ref{eq:omega2}). Finally, (\ref{eq:omega3}) follows in the same way
 using the fact that for three matrices $\mathbf C, \mathbf D, \mathbf E$ we have $[[\mathbf C, \mathbf D], \mathbf E]^\top= [[\mathbf C^\top, \mathbf D^\top], \mathbf E^\top]$. We finish the proof noting that the validity of (\ref{eq:28}) follows from Remark \ref{EasyMagnus}. 
\end{proof}

\section{Examples of time inhomogeneous polynomial processes}\label{sec:examples}

The goal of this section is to present some further examples of polynomial processes and to show how the results of Section \ref{sec:computation} can be used to compute their moments.  We restrict ourselves to one-dimensional examples and processes with state space $E= \R$ or $E=\R_+$.

\subsection{Brownian motion with drift}\label{sec:brown}
Let $T>0$. We consider the Markov process satisfying
\begin{equation}\label{easyexample}
X_t=\int_0^ta(u)du +W_t, \quad 0 \leq t \leq T
\end{equation}
where $W_t$ is a Brownian motion and $a$ being a continuous function. Let us first show directly that this process is polynomial by computing the associated family of transition operators. 

To this end let $A(t):=\int_0^ta(u)du$. Then for $0 \leq s \leq t \leq T$ and functions $f$ in the domain of $P_{s,t}$ we can compute
\begin{eqnarray}\label{moments}
\nonumber	P_{s,t}f(x)&=&\mathbb{E}\left[f\left(X_{t}\right)|X_s=x\right]\\
\nonumber	&=&\mathbb{E}\left[f\left(X_{t}-X_s+x\right)\right]\\
\nonumber	&=&\mathbb{E}\left[f\left(A(t)+W_{t}-A(s)-W_s +x\right)\right]\\
\nonumber	&=&\mathbb{E}[f(A(t)-A(s)+(W_{t}-W_s)+x)]\\
	&=&\int_{\mathbb{R}}f\left(A(t)-A(s)+x+y\right)\phi\left(\frac{y}{t-s} \right) dy,
\end{eqnarray}
where $\phi$ denotes the density of the standard normal distribution. This is the simplest  expression we can find to express the operator $P_{s,t}$ acting on a general function $f$. Now we need to check what happens if $f$ is in $\ccP_k(\R)$. To this end, let us denote by
$$ f_0(x)=1, f_1(x)=x, \ldots, f_k(x)=x^k$$
the canonical basis of $\ccP_k(\R)$. Moreover, for fixed $t > s$ we denote by $m_i$ the $i$-th moment of the $N(0,(t-s))$ distribution (with density $\phi_{(0,t-s)}$). In this case, the odd moments are equal to $0$ and the even moments can be expressed as $$(t-s)^p(p-1)!$$ for $p=2n, n\in\mathbb{N}$.

\begin{proposition}
For $k\geq 0$, $0 \leq s \leq t \leq T$ and $x \in E$ we have:
\begin{equation}
  \label{eq:34}
P_{s,t}f_k(x) =\sum_{i=0}^{k} {{k}\choose{i}}(A(t)-A(s)+x)^{k-i}m_{i},  
\end{equation}
where $m_i$ is the $i$-th moment of the $\cN(0,(t-s))$ distribution. In particular, the process $X$ defined through (\ref{easyexample}) is polynomial.
\end{proposition}
\begin{proof}
For $k=0$ both the left- and right-hand side of (\ref{eq:34}) are $0$. Furthermore, for $k \in \N$ we compute directly 
	\begin{align*}
 P_{s,t}f_{k}(x)&= \int_{\mathbb{R}}(A(t)-A(s)+x+y)^{k}\phi_{(0,t-s)}(y)dy\\
	&=  \sum_{i=0}^k {{k}\choose{i}}(A(t)-A(s)+x)^{k-i} \int_{\mathbb{R}} y^{i}\phi_{(0,t-s)}(y)dy	\\
&= \sum_{i=0}^k {{k}\choose{i}}(A(t)-A(s)+x)^{k-i}m_i.
\qedhere
	\end{align*}
\end{proof}

We continue to consider the Brownian motion with drift given in (\ref{easyexample}). As we have seen, this process is simple enough to compute its moments directly. However, for more complicated polynomial processes this is usually not the case and one has to use the corresponding family of infinitesimal generators (which is easier to obtain) and the results of Section \ref{sec:computation} to compute the moments. In the following we would like to sketch how this can be done.\\

By the It\^o formula, the family of generators of (\ref{easyexample}) is given by
\begin{equation}\label{GenEEx}
	\ccH_tf(x)=a(t)\frac{df(x)}{dx} + \frac{1}{2}\frac{d^2f(x)}{dx^2}, \quad t \geq 0.
\end{equation}

\pagebreak

\begin{lemma}
The family $(\ccH_t)_{t\geq 0}$ commutes. 
\end{lemma}
\begin{proof}
We have 
\begin{eqnarray*}
\ccH_t(\ccH_sf)(x)=a(t)\left( a(s)\frac{d^2f(x)}{dx^2} + \frac{1}{2}\frac{d^3f(x)}{dx^3} \right) + \frac{1}{2}\left(a(s)\frac{d^3f(x)}{dx^3} + \frac{1}{2}\frac{d^4f(x)}{dx^4} \right)\\
\ccH_s(\ccH_tf)(x)=a(s)\left( a(t)\frac{d^2f(x)}{dx^2} + \frac{1}{2}\frac{d^3f(x)}{dx^3} \right) + \frac{1}{2}\left(a(t)\frac{d^3f(x)}{dx^3} + \frac{1}{2}\frac{d^4f(x)}{dx^4} \right)
\end{eqnarray*}
and a short inspection shows that these two terms coincide whenever they are defined.  
\end{proof}

We again choose the standard basis $$ f_0(x)=1, f_1(x)=x, \ldots, f_m(x)=x^m$$
of $\ccP_m(\R)$ and compute the corresponding matrix representation of $\ccH_t$. Since $\ccH_t f_0 = 0, \ccH_t f_1 = a(t)$ and 
$$ \ccH_t f_k = \frac{k(k-1)}{2}\cdot f_{k-2} + k \cdot a(t) \cdot f_{k-1} , \quad k \geq 2,$$
the result looks as follows  (all entries not shown are $0$):
$$ \mathbf H_t=
\begin{pmatrix}
0 & a(t) & 1 &  & & & & \\
 & 0 & 2 a(t) & 3 &  & & &\\
 & &  0 & 3a(t) & 6 & & & \\
 & &  & \ddots & \ddots & \ddots & & \\
 & &  &  & \ddots & \ddots & \ddots& \\
 & &   &  &  & \ddots & \ddots  & \frac{m(m-1)}{2} \\
& &   &  &  & & 0  & ma(t) \\
 & &   &  &  & & & 0 \\
\end{pmatrix} \in \R^{(m+1) \times (m+1)}$$

Also the matrix family $(\mathbf H_t)$ commutes, so by Theorem \ref{thm:magnus2} the representing matrix of $P_{s,t}$ is given by
$$ \mathbf P_{s,t} = e^{\int_s^t \mathbf H_u du}.$$
Since the coefficients $(y_0,\ldots,y_m)$ in the expansion $ P_{s,t} f_k = y_0 f_0 + \ldots +y_m f_m$ are given by the $k$-th column of $\mathbf P_{s,t}$ we can thus compute (using a slightly formal notation)
\begin{eqnarray*}
  \mathbb{E}\left[(X_{t})^k|X_s=x\right] &=& P_{s,t}f_k(x)\\
&=& (1,x,\ldots,x^m) \mathbf P_{s,t} (0,...,0,1,0,...,0)^T,\\
&=& (1,x,\ldots,x^m) e^{\int_s^t \mathbf H_u du} (0,...,0,1,0,...,0)^T,
\end{eqnarray*}  
where the '$1$' in the vector on the right is in the $k$-th spot (starting at $0$). For instance, if we want to obtain the first moment of (\ref{easyexample}) it suffices if we choose $m=k=1$ in the previous procedure. We obtain 
$$\mathbf H_t=\begin{pmatrix} 0& a(t)\\ 0 &0 \end{pmatrix}$$
and
$$  e^{\int_s^{t}\mathbf H_u du} = \exp\left( \begin{pmatrix}
  0 & A(t)-A(s) \\
  0 & 0 
\end{pmatrix}\right) =
\begin{pmatrix} 1&0\\ 0&1 \end{pmatrix} + \begin{pmatrix} 0& A(t)-A(s)\\ 0 &0 \end{pmatrix},$$
where $A(t) = \int_0^t a(u) du$. Hence, 
\begin{eqnarray}
\nonumber \mathbb{E}\left[X_{t}|X_s=x\right] = (1,x) 
  \begin{pmatrix}
    1 & A(t)-A(s) \\
  0 & 1     
  \end{pmatrix}
\begin{pmatrix}
 0 \\
 1
 \end{pmatrix}
=  A(t)-A(s) + x,
\end{eqnarray}
which coincides, as expected, with the result obtained previously.\\


In the above example the computation was straightforward thanks to the commutativity of the matrices $(\mathbf H_t)_{t\geq 0}$. In the next sections we will consider examples where this is not the case. 

\subsection{Ornstein-Uhlenbeck processes} 

The Ornstein-Uhlenbeck processes considered in this section are polynomial as we will formally prove in Section \ref{subsec:polydif}. Here we will just concentrate on their family of generators.\\

(i) We consider the following mean reversion process:
\begin{equation} dX_t=(\theta t-X_t)dt + dW_t, \quad 0 \leq t \leq T,
\end{equation}
where $\theta\in\mathbb{R}$. The infinitesimal generator looks as follows:
	\begin{equation}\label{InfGenOU}
	\ccH_tf(x)=(\theta t-x)\frac{df(x)}{dx} + \frac{1}{2}\frac{d^2f(x)}{dx^2}.
\end{equation}
Being interested in the non-commutativity of this family we compute
	\begin{eqnarray*} 
	\ccH_t(\ccH_sf)(x)&=&(\theta t-x)\left( \theta s\frac{d^2f(x)}{dx^2} - x\frac{d^2f(x)}{dx^2} - \frac{df(x)}{dx} + \frac{1}{2}\frac{d^3f(x)}{dx^3} \right) \\
	&& + \frac{1}{2}\left( \theta s \frac{d^3f(x)}{dx^3} - x \frac{d^3f(x)}{dx^3} - 2 \frac{d^2f(x)}{dx^2} + \frac{1}{2}\frac{d^4f(x)}{dx^4}\right).
\end{eqnarray*}
For the sake of clarity, we show this composition in the case of $f=f_1 \in \ccP_1(\mathbb{R})$ (recall that $f_1(x)=x$).
\begin{eqnarray*}
\ccH_t(\ccH_sf_1)(x)&=& x-\theta t\\
\ccH_s(\ccH_tf_1)(x)&=& x-\theta s.
\end{eqnarray*}
Therefore it is clear that the family of infinitesimal generators does not commute.

\section{Computation of transition operators}\label{sec:computation}
In the following we will indicate the computation of the corresponding transition operators $P_{s,t}$ when acting on $\ccP_2(\R)$. The representing matrix of $\ccH$ with respect to the basis $$f_0(x)=\varepsilon^{-1}, f_1(x)=\varepsilon^{-1} x, f_2(x)= \varepsilon^{-1} x^2$$ 
with $\varepsilon >0$ is given by
\begin{equation}
\mathbf H_t= \varepsilon
\begin{pmatrix}
0 & \theta t & 1 \\
0 & -1 & 2 \theta t \\
0 & 0 & -2
\end{pmatrix}.
\end{equation}
In Appendix A of \cite{Agoitia2017} it is shown that
$$\left\|\mathbf H_t\right\|_2=\varepsilon\sqrt{\frac{5\theta^2t^2 +6 +\sqrt{(3\theta^2t^2 +4)^2 + 4\theta^2t^2}}{2}}.$$
We are always able to choose $\varepsilon$, such that
$$ \int_0^T \|\mathbf H_u\|_2 du < \pi.$$
Then, the corresponding matrix $\mathbf P_{s,t}$ is given by 
$$ \mathbf P_{s,t} = e^{\mathbf \Omega(s,t)}$$
and the  Magnus series $\mathbf \Omega(s,t)=\sum_{k=1}^\infty \mathbf \Omega_k(s,t)$ is absolutely convergent by Theorem \ref{thm:magnus2}. The first three terms of the Magnus series look as follows:
\begin{equation*}
	\mathbf\Omega_1(s,t)=\int_s^t \mathbf H_u du= \varepsilon\begin{pmatrix} 
0 & \theta(t-s) & t-s \\
0 & s-t &2\theta(t-s) \\ 
0 & 0 & 2(s-t)
\end{pmatrix},
\end{equation*}
\begin{eqnarray*}
	\mathbf \Omega_2(s,t)&=& -  \frac{1}{2}\int_s^t du\int_s^u[ \mathbf H_u, \mathbf H_v] dv\\
	&=& -\varepsilon^2 \begin{pmatrix} 
0 & \frac{\theta}{12} (t^3-3st^2+s^2t-4s^3) & 0\\ 
0 & 0 & \frac{\theta}{6}(t^3-3st^2-s^2t-4s^3)\\ 
0 & 0 & 0 \end{pmatrix},
\end{eqnarray*}
and
\begin{eqnarray*}
\mathbf	\Omega_3(s,t)&=& \frac{1}{6}\int_s^t du \int_s^udv \int_s^v\left([\mathbf H_u,[\mathbf  H_v, \mathbf  H_w]+[\mathbf  H_w,[\mathbf  H_v,\mathbf H_u]]\right)dw\\
	&=& \frac{\varepsilon^3}{6}\begin{pmatrix} 0 & a_1 & a_2\\ 0 & 0 & a_3\\ 0 & 0 & 0 \end{pmatrix},
\end{eqnarray*}
where $$a_1=-\frac{7}{24}\theta t^4 - \left(\frac{23}{24}\theta - \frac{2}{3}\right)s^4 + \left(\frac{9}{6}\theta +\frac{1}{2}\right)s^3t + \left(\frac{1}{6}-\frac{\theta}{6}\right)st^3 - \frac{1}{2}s^3,$$
$$a_2=-\frac{\theta}{3}t^3 + {\theta}{3}s^3- \theta s^2t + \theta st^2, $$
and
$$a_3=\frac{1}{24}t^4 - \left(\frac{1}{8} - \frac{5}{6}\theta\right)s^4 + \frac{2}{3}\theta st^3 - \theta s^2t + \left(\theta-\frac{1}{6}\right)s^3t + \theta s^3.$$
Now the matrix $e^{\mathbf \Omega_1(s,t) + \mathbf \Omega_2(s,t) + \mathbf \Omega_3(s,t)}$ can be used as an approximation to $\mathbf P_{s,t}$.

\subsection{Ornstein-Uhlenbeck process} With a short computation under an Ornstein-Uhlenbeck process we intend to show that the commutativity of the operators can even depend on the space. We consider the following stochastic process:
	\[ dX_t= tX_tdt + dW_t, \quad 0 \leq t \leq T.
\]
Its family of infinitesimal generators is given by
\[\ccH_tf(x)=tx\frac{df(x)}{dx} + \frac{1}{2}\frac{d^2f(x)}{dx^2}.
\]
A short computation shows that then $\ccH_s(\ccH_tf)(x)$ is given by 
\begin{eqnarray*}
sx\left( t\frac{df(x)}{dx}+ tx\frac{d^2f(x)}{dx^2} + \frac{1}{2}\frac{d^3f(x)}{dx^3} \right) + \frac{1}{2}\left( tx\frac{d^3f(x)}{dx^2} + 2 t\frac{d^2f(x)}{d^2x} + \frac{1}{2}\frac{d^4f(x)}{d^4x} \right).  
\end{eqnarray*}
When we consider $\mathcal H_s$ on $\ccP_1(\mathbb{R})$ and evaluate the composition on $f_1(x)=x$ we obtain:
\begin{eqnarray*} 
\ccH_s(\ccH_tf_1)(x)&=& stx\\
\ccH_t(\ccH_sf_1)(x)&=& tsx,
\end{eqnarray*}
and then it commutes (and it certainly also commutes on $f_0$ where both terms vanish). But if we consider $\mathcal H_s$ on $\ccP_2(\mathbb{R})$ and evaluate the composition on $f_2(x)=x^2$, then
\begin{eqnarray*}
\ccH_s(\ccH_tf_2)(x)&=& 4tsx^2 + 2t\\
\ccH_t(\ccH_sf_2)(x)&=& 4stx^2 +2s,
\end{eqnarray*}
so the family does not commute on $\ccP_2(\R)$.

\subsection{Jacobi process}
We consider one last type of stochastic process, the Jacobi process, which is a diffusion with a barrier level $b\in\mathbb{R}_+$. It corresponds to the stochastic differential equation	
	\[ dX_t=a(t) +\sqrt{X_t(b-X_t)}dW_t,
\]
where $a(t)$ is a continuous function and $W_t$ is a Brownian motion. We will see in the next chapter that this process is polynomial (see Section \ref{subsec:polydif}). For every $t\geq 0$ its infinitesimal generator is given by

	\[\ccH_tf(x)=a(t)\frac{df(x)}{dx} + \frac{1}{2}x(b-x)\frac{d^2f(x)}{dx^2}.
\]
Again focusing on commutativity of this family we obtain  	
{\small 
\begin{eqnarray*}
&&\ccH_t(\ccH_sf)(x)=a(t)\left( a(s)\frac{d^2f(x)}{dx^2} + \frac{1}{2}\left( (b-2x)\frac{d^2f(x)}{dx^2} + x(b-x)\frac{d^3f(x)}{dx^3} \right) \right)\\
&& +\frac{1}{2} x(b-x)\left( a(s)\frac{d^3f(x)}{dx^3} + \frac{1}{2}\left( -2\frac{d^2f(x)}{dx^2} + 2(b-2x)\frac{d^3f(x)}{dx^3}+ x(b-x)\frac{d^4f(x)}{dx^4} \right) \right).
\end{eqnarray*}
}
We observe that since this expression only contains derivatives of order 2 and higher the family $(\mathcal H_s)$ commutes on $\ccP_1(\R)$. However, for $f_2(x)=x^2$ we obtain
\begin{eqnarray*}
	\ccH_t(\ccH_s)f_2(x)&=&a(t)(2a(s)+ bx - 2x^2)- x(b-x)\\
	\ccH_s(\ccH_t)f_2(x)&=&a(s)(2a(t)+ bx - 2x^2)- x(b-x),
\end{eqnarray*}
so the family $(\mathcal H_s)$ does not commute on $\ccP_2$, making the computation of moments of order $\geq 2$ considerably more difficult than the computation of the first moment.\\

\begin{appendix}
\section{Additional results}\label{sec:appendix}
Recall the definition of the space $\tilde{\ccP}_m(S)$ time-inhomogeneous polynomials being continuously differentiable and having degree of at most $m$ from Equation \eqref{PolSpace}. The following proposition shows that this space is a Banach space under an appropriate norm. 

\begin{proposition}
\label{BanachPk}
For a polynomial $p(s,t,x)=\sum_{|\bk|=0}^m \alpha_{\bk}(s,t)x^\bk \in \tilde{\ccP}_m(S)$  we define the norm   	
\[ \left\VERT p \right\VERT_m  := \sum_{|\mathbf{k}|=0}^m \left(\max_{(s,t) \in\Delta}\left|\alpha_{\mathbf{k}}(s,t)\right| + \max_{(s,t) \in\Delta}\left|D_1 \alpha_{\mathbf{k}}(s,t)\right| + \max_{(s,t)\in\Delta}\left|D_2 \alpha_{\mathbf{k}}(s,t)\right| \right),\] 
Then the space $(\tilde{\ccP}_m(S),\parallel \cdot \parallel_m)$ is complete.
\end{proposition}

\begin{proof}
Let $(f_n)$ be a Cauchy sequence in $\widetilde{\ccP}_m(S)$ with representation
$$f_n(s,t,x)= \sum_{|\mathbf{k}|=0}^m \alpha_{\mathbf{k}}^{n}(s,t)x^{\mathbf k}, \ \ \ \forall n\in\mathbb{N}.$$ 
Then, for all $\epsilon>0$ there exists $N\in\mathbb{N}$ such that for all $n,p\geq N$ it holds that $\left\VERT f_p-f_n\right\VERT_m<\epsilon$, hence
\begin{eqnarray*}
&&\sum_{|\mathbf{k}|=0}^m\left(\max_{(s,t)\in\Delta}|\alpha_{\mathbf{k}}^{p}(s,t)-\alpha_{\mathbf{k}}^{n}(s,t)| + \max_{(s,t)\in\Delta}|D_1 \alpha_{\mathbf{k}}^{p}(s,t)- D_1 \alpha_{\mathbf{k}}^{n}(s,t)| \right. \\ 
&& \left. +\max_{(s,t)\in\Delta}|D_2 \alpha_{\mathbf{k}}^{p}(s,t)- D_2 \alpha_{\mathbf{k}}^{n}(s,t)| \right)<\epsilon.  
\end{eqnarray*} 
As a consequence, we obtain for all $0 \leq |\mathbf{k}| \leq m$: 
\begin{eqnarray*}
\|\alpha_{\mathbf{k}}^{p}-\alpha_{\mathbf{k}}^{n}\|_{C^1} &:=& \max_{(s,t)\in\Delta}|\alpha_{\mathbf{k}}^{p}(s,t)-\alpha_{\mathbf{k}}^{n}(s,t)| + \max_{(s,t)\in\Delta}|D_1 \alpha_{\mathbf{k}}^{p}(s,t)- D_1 \alpha_{\mathbf{k}}^{n}(s,t)| \\
&& +\max_{(s,t)\in\Delta}|D_2 \alpha_{\mathbf{k}}^{p}(s,t)- D_2 \alpha_{\mathbf{k}}^{n}(s,t)| <\epsilon.
\end{eqnarray*}
This shows that for all $0 \leq |\mathbf{k}| \leq m$ the sequence $(\alpha_{\mathbf{k}}^{n})_{n\in\mathbb{N}}$ is a Cauchy sequence in $(C^1(\Delta),\|.\|_{C^1})$. Since this space is complete, 
there exists for every $0 \leq |\mathbf{k}| \leq m$  an $\alpha_{\mathbf{k}}\in C^1(\Delta)$ such that        
$ \|\alpha_{\mathbf{k}}^{n}-\alpha_{\mathbf{k}}\|_{C^1}\stackrel{n\rightarrow \infty}{\rightarrow}0.$
Set $f(s,t,x):=\sum_{|{\mathbf{k}}|=0}^m \alpha_{\mathbf{k}}(s,t)x^{\mathbf{k}}$ for $(s,t,x) \in \widetilde{E}$. Then $f \in \widetilde{\ccP}_m(S)$ since $\alpha_{\mathbf{k}}\in C^1(\Delta)$ and
\begin{eqnarray*}
\left\VERT f_n-f\right\VERT_m &=& \sum_{|{\mathbf{k}}|=0}^m\left(\max_{(s,t)\in\Delta}|\alpha_{\mathbf{k}}^{n}(s,t)-\alpha_{\mathbf{k}}(s,t)| + \max_{(s,t)\in\Delta}|D_1 \alpha_{\mathbf{k}}^{n}(s,t)- D_1 \alpha_{\mathbf{k}}(s,t)| \right. \\
&& \left.+
 \max_{(s,t)\in\Delta}|D_2 \alpha_{\mathbf{k}}^{n}(s,t)- D_2 \alpha_{\mathbf{k}}(s,t)| \right) \\
&=& \sum_{|\mathbf{k}|=0}^m  \|\alpha_{\mathbf{k}}^{n}-\alpha_{\mathbf{k}}\|_{C^1}\stackrel{n\rightarrow\infty}{\rightarrow} 0.
\end{eqnarray*}
So the Cauchy sequence $(f_n)$ converges to $f$ and hence $\widetilde{\ccP}_m(S)$ is complete.
\end{proof}

The following lemma  shows that convergence of polynomials with respect to the norm $\|.\|_m$, defined in Equation \ref{def:normm}, is equivalent to pointwise convergence at every point. 
 \begin{lemma}\label{lem:weak} 
 Let $U \subset \mathbb R^d$ and for every $u \in U$ let $p_u \in \ccP_m(S)$. Then for $u_0 \in U$ the following are equivalent:
 \begin{enumerate}[i)]
     \item $\|p_u-p_{u_0}\|_m \to 0$ for $u \to u_0$.
     \item For every $x \in S$: $|p_u(x)-p_{u_0}(x)| \to 0$ for $u \to u_0$.
 \end{enumerate}

 \end{lemma}
 \begin{remark}\label{rem:dual}
 In the proof of this proposition we denote the dual space of a finite dimensional vector space $V$ by $V'$.
We recall that in every finite-dimensional normed space $(V, \|.\|)$ the weak- and norm topologies coincide (see  \cite{Kreyszig89}, Section 4.8). This actually means that for a function $f:U \subset \R^d \to V$ the following are equivalent:
\begin{enumerate}
    \item[a1)] $\|f(u)-f(u_0)\| \to 0$ for $u \to u_0$.
    \item[a2)] For every $l \in V'$: $|l(f(u))-l(f(u_0))| \to 0$ for $u \to u_0$.
\end{enumerate}
 \end{remark}

\begin{proof}
Using the notation of the previous remark we choose  $(V, \|.\|)=(\ccP_m(S), \|.\|_m)$ and  $f: U \to V, f(u)=p_u$.

 $i) \Rightarrow ii)$: This follows immediately from $a1) \Rightarrow a2)$ noting that for every fixed $x \in S$ the functional
 \begin{equation}
   \label{eq:16}
 l_x : \ccP_m(S) \to \mathbb R, \quad l_x(p)=p(x)   
 \end{equation}
is linear, i.e. an element of the dual space $(\ccP_m(S))'$. 

$ii) \Rightarrow i)$: Let $N=\dim(\ccP_m(S))$. Then there exist $N$ different points $x_1,\ldots,x_N \in S$ such that the linear functionals $l_{x_1},\ldots, l_{x_N}$ 
constitue a basis of $(\ccP_m(S))'$: indeed, the linear span of the evaluation functionals $\{ l_x:x \in S\}$ is the dual space $(\ccP_m(S))'$. Since the dimension of a vector space and its dual space coincide we hence find a basis containing $N$ elements. 

For arbitrary $l \in \mathcal (\ccP_m(S))'$  we hence find $\alpha_1, \ldots, \alpha_N \in \R$ such that $l=\alpha_1 l_{x_1} + \ldots + \alpha_N l_{x_N}$. In particular, we obtain that for $u, u_0 \in U$
\begin{eqnarray*}
  |l(f(u)) - l(f(u_0))| &=& |l(p_u)-l(p_{u_0})| = |l(p_u-p_{u_0})| \\
 &=& |\alpha_1 l_{x_1}(p_u-p_{u_0}) + \ldots + \alpha_N l_{x_N}(p_u-p_{u_0})| \\
 &=& |\alpha_1 (p_u-p_{u_0})(x_1) + \ldots + \alpha_N (p_u-p_{u_0})(x_N)| \\
 &\leq & |\alpha_1| |p_u(x_1)-p_{u_0}(x_1)| + \ldots + |\alpha_N| |p_u(x_N)-p_{u_0}(x_N)|
\end{eqnarray*}
and since we assume $ii)$ to hold the right-hand side tends to $0$ for $u \to u_0$. So a2) is satisfied and the implication $a2) \Rightarrow a1)$ shows the validity of $ii) \Rightarrow i)$.
\end{proof}

\begin{lemma}\label{lem:pazy}
Let $(\mathbf A(t))_{0 \leq t <T}$ denote a continuous family of matrices in $\R^{N \times N}$. Then for fixed $s \in [0,T)$ the initial value problem
$$ \left\{
  \begin{array}{cl}
\frac{d}{dt} \mathbf V(s,t)= \mathbf V(s,t) \mathbf A_t, & s < t <T \\
 \mathbf V(s,s)= \mathbf I 
  \end{array}\right.$$
has a unique solution $\mathbf V$ such that $ \mathbf V(s,t)$ is continuously differentiable in both variables and satisfies  
\begin{equation*}
 \frac{d}{ds} \mathbf V(s,t) = - \mathbf A_s \mathbf V(s,t).
\end{equation*}
\end{lemma}
\begin{proof}
 The result is a  direct application of Theorem 5.2 in  \cite{Pazy92}.
\end{proof}

In the next theorem, we state some additonal  properties of the families $(P_{s,t})$ and $(\ccH_s)$.

\begin{proposition}\label{Thm:gen}
Let $X=(X_t)_{0 \leq t \leq T}$ be an $S$-valued $m$-polynomial process with family of transition operators $(P_{s,t})_{(s,t)\in\Delta} $ and family of infinitesimal generators $(\ccH_s)_{0 \leq s < T}$. Then for $0 \leq k \leq m$ the following holds:
 for every $f \in \ccP_k(S)$ the map 
$$[0,T) \ni s \mapsto \ccH_s f \in (\ccP_k(S),\|.\|_k)$$ 
is continuous, i.e. the family $(\ccH_s)_{0 \leq s < T}$ is strongly continuous on $(\ccP_k(S),\|.\|_k)$.
\end{proposition}
\begin{proof}   
Let $0 \leq k \leq m$ and $f \in \ccP_k(S)$. Then $P_{s,t}f(x) = \sum_{|\mathbf{l}|=0}^k \alpha_{\mathbf l}^f(s,t) x^{\mathbf l}, x \in S,$ for some $\alpha_{\mathbf l}^f \in C^1(\Delta)$ and we have seen in Lemma \ref{Lem:Pointwise} that
\begin{equation*}
(\ccH_sf)(x)=\sum_{|\mathbf{l}|=0}^k \partial_2^+\alpha_{\mathbf{l}}^f(s,s)x^{\mathbf{l}}, \quad x \in S,\quad  0 \leq s <T.  
\end{equation*}
Hence for $0 \leq s,r < T$ we can compute
\begin{eqnarray}
\nonumber \left\|\ccH_rf-\ccH_{s}f\right\|_k 
\nonumber &=&\left\|\sum_{|\mathbf{l}|=0}^k\left( \partial_2 \alpha_{\mathbf{l}}^f(r,r) - \partial_t \alpha_{\mathbf{l}}^f(s,s)\right)x^{\mathbf{l}}\right\|_k\\
\nonumber &=& \max_{0\leq |\mathbf{l}| \leq k} \left| \partial_2 \alpha_{\mathbf{l}}^f(r,r) - \partial_2 \alpha_{\mathbf{l}}^f(s,s) \right|.
\end{eqnarray}
The last term tends to $0$ when $r$ tends to $s$ by the continuity of $\partial_2 \alpha_{\mathbf{l}}^f.$ 
\end{proof}

\end{appendix}

\end{document}